\documentclass[11pt,a4paper,reqno]{amsart}%
\usepackage[utf8]{inputenc}
\usepackage{enumitem}
\usepackage{amsfonts, verbatim}
\usepackage{amsmath}
\usepackage{amsthm}
\usepackage{amssymb}
\usepackage[pdftex]{hyperref}
\usepackage{mathrsfs}
\usepackage{graphicx}
\usepackage[initials]{amsrefs}
\usepackage{color}

\usepackage[foot]{amsaddr}

\newcommand{\R}{\mathbb{R}}

\DeclareMathOperator\dv{div}
\DeclareMathOperator\Ric{Ric}
\numberwithin{equation}{section}
\theoremstyle{plain}
\newtheorem{thm}{Theorem}[section]
\newtheorem{lem}[thm]{Lemma}
\newtheorem{proposition}[thm]{Proposition}
\newtheorem{cor}[thm]{Corollary}
\newtheorem{prop}[thm]{Proposition}
\newtheorem{rem}[thm]{Remark}
\theoremstyle{definition}
\newtheorem{defin}[thm]{Definition}

\newcommand{\dive}{\operatorname{div}}
\newcommand{\Hn}{\mathbb{H}^n}

\author[Bonorino]{Leonardo Bonorino$\null^{\; 1,4}$}
\address{$\null^{1}$Universidade Federal do Rio Grande do Sul, IME, Av. Bento Gon\c{c}alves 9500, Porto Alegre, RS, Brazil 91509-900}
\email{$\null^{4}$bonorino@mat.ufrgs.br}

\author[Casteras]{Jean-Baptiste Casteras$\null^{\; 2,5}$}
\address{$\null^{2}$ D\'epartement de Math\'ematiques, Universit\'e Libre de Bruxelles, CP 214, Boulevard du triomphe, B-1050 Bruxelles, Belgium.}
\email{$\null^{5}$jeanbaptiste.casteras@gmail.com}

\author[Klaser]{Patricia Klaser$\null^{\; 3, 6}$}
\address{$\null^{3}$ Departamento de Matem\'atica, Universidade Federal de Santa Maria, Av. Roraima, 1000 - Camobi,
Santa Maria RS 97105-900, Brazil}
\email{$\null^{6}$patricia.klaser@ufsm.br}

\author[Ripoll]{Jaime Ripoll$\null^{\; 1,3,7}$}
\email{$\null^{7}$jaime.ripoll@ufrgs.br}

\author[Telichevesky]{Miriam Telichevesky$\null^{\; 1,8}$}
\email{$\null^{8}$miriam.telichevesky@ufrgs.br}

\begin{document}
\title[On the asymptotic Dirichlet problem]{On the asymptotic Dirichlet problem for a class of mean curvature type partial differential equations}
\maketitle

\date{\{\}}

\begin{abstract}
We study the Dirichlet problem for the following prescribed mean curvature PDE
$$
\begin{cases}
-\operatorname{div}\dfrac{\nabla v}{\sqrt{1+|\nabla v|^{2}}}=f(x,v)
\text{ in }\Omega\\
v=\varphi \text{ on }\partial\Omega.
\end{cases}
$$
where $\Omega$ is a domain contained in a complete Riemannian manifold $M,$
$f:\Omega\times\mathbb{R\rightarrow R}$ is a fixed function and $\varphi$ is a given continuous function on $\partial\Omega$. This is done in three parts. In the first one we consider this problem in the most general form, proving the existence of solutions when $\Omega$ is a bounded $C^{2,\alpha}$ domain, under suitable conditions on $f$, with no restrictions on $M$ besides completeness. In the second part we study the asymptotic Dirichlet problem when $M$ is the hyperbolic space $\mathbb{H}^n$ and $\Omega$ is the whole space. This part uses in an essential way the geometric structure of $\mathbb{H}^n$ to construct special barriers which resemble the Scherk type solutions of the minimal surface PDE. In the third part one uses these Scherk type graphs to prove the non existence of isolated asymptotic boundary singularities for global solutions of this Dirichlet problem. 

\end{abstract}

\


\


\

\section{Introduction}

A natural way of finding bounded entire solutions to a partial differential equation on a Cartan-Hadamard manifold (complete, simply connected Riemannian manifold with nonpositive sectional curvature) is by solving the asymptotic Dirichlet problem with a prescribed  boundary data given at infinity. This problem has been extensively studied for the Laplace equation mostly motivated by the Green-Wu conjecture which asserts the existence of bounded non constant harmonic functions on a Cartan-Hadamard manifold under certain growth and decay conditions on the sectional curvature (see \cite{GW}, \cite{SY}). In the last years the asymptotic Dirichlet  problem has been studied for other partial
differential equations such as the $p-$Laplace (\cite{Ho}) and the minimal
surface equation (\cite{GR}, \cite{CHR}, \cite{CHR2}, \cite{RT}).

In our paper we study the Dirichlet problem for the following prescribed mean curvature PDE
\begin{equation}
\begin{cases}
-\operatorname{div}\dfrac{\nabla v}{\sqrt{1+|\nabla v|^{2}}}=f(x,v)
\text{ in }\Omega\\
v=\varphi \text{ on }\partial\Omega.
\end{cases}
\label{eq-fgraph}
\end{equation}
where $\Omega$ is a domain contained in a complete Riemannian manifold $M,$
$f:\Omega\times\mathbb{R\rightarrow R}$ is a fixed function satisfying some
conditions and $\varphi$ is a given continuous function on $\partial\Omega$.

The objective of this paper is threefold: first, to investigate the existence of solutions of \eqref{eq-fgraph} when $\Omega$ is bounded; second, to study the asymptotic Dirichlet problem in the case where $M$ is the hyperbolic space $\Hn$ and, third, to study the existence or not of isolated asymptotic boundary
singularities for the solutions to the problem discussed in the second step. We make some comments on each of such problems.

Problem \eqref{eq-fgraph} in the case of bounded domains and where  $f$ is a constant or a function depending only on $x$ is a classical topic of study in the Euclidean geometry which, more recently,  has been studied in the Riemannian setting. The theory has now reached a well developped stage and problem \eqref{eq-fgraph} is completely
solved for a large class of PDE's on bounded domains of  complete Riemannian manifolds (see \cite{RT} and references therein for the case where $f=0$ and \cite{dajczer2008killing} for $f$ depending on $x$ and for the mean curvature PDE). Concerning the case where $f$ depends on $x$ and $u,$ a uniqueness result for \eqref{eq-fgraph} has been obtained in \cite{AMP} provided that $\Omega$ is bounded and $f(x,\cdot)$, for a fixed $x\in \Omega$, is nonincreasing. Here, we will provide an existence result for the problem \eqref{eq-fgraph} under the same assumption on $f$ and some geometric assumptions on the domain which are well know to be necessary when dealing with the mean curvature operator.  We shall use the method of a priori estimates for proving the existence of classical solutions. In the case where $f$ depends only on $x$ and for the mean curvature PDE our existence results recover the ones mentioned above.

It is natural to investigate the asymptotic Dirichlet problem once the solvability of \eqref{eq-fgraph} has been established in bounded domains for continuous boundary data. Despite the vast literature on this problem most of the results deal only with homogeneous PDEs of the form $\operatorname{div}(a(|\nabla u|)\nabla u)=0.$  We study here the existence of solutions to the inhomogeneous asymptotic problem \eqref{eq-fgraph} but only in the hyperbolic space. The reason for confining ourselves to this space comes from the construction of barriers at infinity, which are fundamental to prove the continuous extension to the asymptotic boundary of a prospective entire bounded solution. As it is well known from several works, the construction of such barriers is closely related to the existence of Scherk type super and sub solutions to \eqref{eq-fgraph} (see definition \ref{def-Scherkprob}). Due the inhomogeneous part of the PDE \eqref{eq-fgraph}, the geometric structure of the background manifold seems to be  fundamental for the construction of such sub (super) solutions. Indeed, one uses here strongly the symmetries of the hyperbolic geometry to construct such barriers in a quite explicit way (see Secion \ref{sec-STSolutions}). This is the largest part of the paper which, despite being elementary, is more involved. The construction of such barriers requires a decay on $|f(x,u)|$ as $x$ goes to the asymptotic boundary as well as some global assumption on $f(x,u)$ (see \ref{propphi1} and \ref{propphi2} for a precise statement). 
The fact that some sort of decay of $|f|$ at infinity is necessary follows from the geometric nature of the mean curvature operator (see \cite{PigolaRigoliSetti}). Indeed, an application of the tangency principle gives an obstruction to the mean curvature of a solution by comparing the mean curvature of its graph with the mean curvature of geodesic spheres. Let us point out that existence results have been obtained in \cite{CHHfmin} for a related equation on more general Cartan-Hadamard manifolds with a different method not using Scherk type solutions. 

The existence or not of interior or boundary singularities for the minimal surface equation in the Euclidean space is a classical topic of study. In \cite{BR} the authors extended this study to the asymptotic Dirichlet problem  in a Riemanian manifold and to a large class or partial differential equations which includes the $p-$Laplace and the minimal surface equation. In particular, they prove that isolated singularities at the asymptotic boundary of solutions of the minimal surface equation are removable. We here obtain the same result to the inhomogeneous PDE \eqref{eq-fgraph} in the hyperbolic space.
 
We state precisely our main results. Let us begin with our existence results on bounded domains, proved in Section \ref{sec-genExist}.

\begin{thm}
\label{thm-existence}
Let $\Omega \subset M$ be a bounded $C^{2,\alpha}$ domain in a complete Riemannian manifold $M$ and let $f\in C^{1}\bigl(\overline{\Omega}\times \mathbb{R}\bigr)$ be given. 
Suppose there is a constant $F$ such that
\begin{equation}
|f(x,t)|\leq F, \text{ and } f_t(x,t)\le 0 \text{ for all }\left(x,t\right)  \in\overline{\Omega}\times\mathbb{R}
\label{eq-f-form}
\end{equation}
and
\begin{equation}
\operatorname*{Ric}\nolimits_{\Omega}\geq-\dfrac{F^{2}}{n-1},\quad
H_{\partial\Omega}\geq F, \label{curvboundary}
\end{equation}
where
$\operatorname*{Ric}\nolimits_{\Omega}$ stands for the Ricci curvature of
$\Omega$ and $H_{\partial\Omega}$ for the inward mean curvature of
$\partial\Omega$. 
Then the Dirichlet problem \eqref{eq-fgraph}
is solvable
for all $\varphi\in C^{0}(\partial\Omega)$. If $\varphi\in C^{2,\alpha}(\overline
{\Omega}),$ then the solution is also in $C^{2,\alpha}(\overline{\Omega}).$
\end{thm}

We observe that Theorem \ref{thm-existence} extends \cite[Theorem 1]{dajczer2008killing} to the case where the function $f$ depends also on $u$.

In Section \ref{sec-STSolutions}, we construct Scherk type sub and super solutions (see Definition \ref{def-Scherkprob} and Theorem \ref{s}) which are used to prove the following results:

\begin{thm}\label{thm-Scherk}
Suppose that $f\in C^{1}(\mathbb{H}^n\times\mathbb{R})$ satisfies
$f_t(x,t)\le 0$  in $\Hn\times\mathbb{R}$ and condition  \ref{propphi1} for a function $\phi(r)\leq (n-1)\coth(r).$ Then the asymptotic
Dirichlet problem 
\begin{equation}
\begin{cases}\label{eq-asymDP}
-\operatorname{div}\dfrac{\nabla v}{\sqrt{1+|\nabla v|^{2}}}=f(x,v)
\text{ in }\mathbb{H}^n\\
v=\varphi \text{ on }\partial_{\infty}\mathbb{H}^n
\end{cases}
\end{equation}
is solvable for any $\varphi\in C^{0}(\partial_{\infty}\mathbb{H}^n)$. Moreover, 
the assumption $\phi(r)\leq (n-1)\coth(r)$ can be replaced by condition \ref{propphi2} on $f$. 
\end{thm}

Let us observe that, even using Perron's method, we are not able to prove Theorem \ref{thm-Scherk} only assuming some asymptotic decay condition on $f$ i.e. condition \ref{propphi1}. Beyond $f_t \leq 0$, we also need some global assumptions on $f$, precisely, we require either that $|f(x,t)|\leq (n-1)\coth (r(x))$, for all $t\in \R$ and  $x\in \mathbb{H}^n$, or condition \ref{propphi2}. Notice that these assumptions guarantee the solvability of some Dirichlet problem on balls $B_R (o)$.
To see their importance, first observe that for $H > n-1$ there are hemispheres of mean curvature $H$ which are graphs of functions $u: B_R(o)\to \R$ with infinite gradient at the boundary. Consider the case $f(x,t) = H$ for $r(x) < R$ and $t \in \mathbb{R}$. Then $f(x,t)$ does not satisfy either $|f(x,t)|\leq (n-1)\coth (r(x))$ or condition \ref{propphi2}, even if condition \ref{propphi1} and $f_t \le 0$ hold. We can prove using a comparison argument with these hemispheres that equation $Q(v)=f(x,v)$ has no solution in any domain containing $B_R(o)$.

\

We conclude this paper by generalizing partially Theorem 1.3 of \cite{BR}:

\begin{thm}
Suppose that $f\in C^{1}(\mathbb{H}^n\times\mathbb{R})$ satisfies $f_t(x,t)\le 0$  in $\Hn\times\mathbb{R}$, \ref{propphi1} and \ref{propphi2}. Let $p_1, \dots, p_k \in\partial_{\infty}\mathbb{H}^n$ and $\varphi \in C^0(\partial_\infty \Hn)$ be given.  If $u \in C^2(\mathbb{H}^n)\cap C^0(\overline{\mathbb{H}^n} \backslash \{p_1, \dots, p_k \})$ is a solution to the Dirichlet problem
\begin{equation}
\begin{cases}
-\operatorname{div}\dfrac{\nabla v}{\sqrt{1+|\nabla v|^{2}}}=f(x,v)
\text{ in }\mathbb{H}^n\\
v=\varphi \text{ on }\partial_{\infty}\mathbb{H}^n\backslash\{p_1,\,\dots,\, p_k\},
\end{cases}
\end{equation}
 then $u \in C^0\bigl(\overline{\mathbb{H}^n}\bigr)$.
\label{ContinuityOfSolutionWithSing}
\end{thm}

\section{A general existence theorem in bounded domains of Riemannian manifolds}\label{sec-genExist}

In this section we prove Theorem \ref{thm-existence}. As it is well-known from the theory of second order quasi linear elliptic PDE (see \cite{GilTru}) Theorem \ref{thm-existence} will follow once we get a priori height and gradient estimates for solutions of \eqref{eq-fgraph}.

From \cite[Theorem 1]{dajczer2008killing} given $\varphi \in C^{2,\alpha}(\partial\Omega),$ under the hypothesis of Theorem \ref{thm-existence}, there are functions $w^+, w^-\in C^{2,\alpha}(\overline{\Omega})$ such that $Q(w^+)=F,$ $Q(w^-)=-F$  in $\Omega$ and $w^+, w^-$ are equal to $\varphi$ on $\partial\Omega$, where \begin{equation}\label{def-opQ}
Q(v)=-\dive\left(\dfrac{\nabla v}{\sqrt{1+|\nabla v|^{2}}}\right).
\end{equation}

Hence if $u\in C^1(\overline\Omega)$  is a solution to \eqref{eq-fgraph}, it holds that $w^-\leq u \leq w^+$ in $\Omega$ and these  functions coincide on $\partial \Omega.$ Therefore,
taking $$C=\max\{\sup_\Omega |w^+|,\sup_\Omega |w^-|, \sup_{\partial\Omega} |\nabla w^+|, \sup_{\partial\Omega} |w^-|\}= C(\varphi, \Omega, F),$$ we have 

\begin{lem}
\label{lem-altEgradbordo} Suppose that $f\in C^{1}\bigl(\overline{\Omega}\times \mathbb{R}\bigr)$ satisfies \eqref{eq-f-form} and that
\eqref{curvboundary} holds. Let $u \in C^{2}(\Omega)\cap C^1(\overline{\Omega})$  be a
solution of \eqref{eq-fgraph}. Then, there exists a constant
\[
C= C(\varphi, \Omega, F)\]
such that
\[\sup_{\Omega} |u| \le C \text{ and }
\sup_{\partial\Omega} |\nabla u| \le C.
\]
\end{lem}

We also need local and global gradient estimates as stated below.

\begin{proposition}
Consider problem \eqref{eq-fgraph} in $\Omega \subset M$, a bounded $C^{2,\alpha}$ domain. Suppose that $f\in C^1(\overline{\Omega} \times \mathbb{R})$  is a  function for which there is a constant $A,$ such that 
\begin{equation}
 |f(x,t)| \le A, \quad |f_x(x,t)| \le A  \text{ and }  f_t(x,t) \le 0
\label{boundOf-f-form}
\end{equation}
for any $(x,t) \in \overline{\Omega} \times \mathbb{R}$. Let $u\in C^3(\Omega)$ be a solution of \eqref{eq-fgraph}. 
\begin{enumerate}
\item[(a)] If $u \in C^1(\overline{\Omega})$, then there is $L=L(\max_\Omega u, \max_{\partial \Omega}|\nabla u|, A,\Ric_{\Omega} ) > 0$ such that, for any $x \in \Omega,$
$$ |\nabla u(x)| \le L.$$
\item[(b)] For any normal geodesic ball $B_R(x_0) \subset \subset \Omega$, there exists \\ $L=L(\max_\Omega u, A, \Ric_{\Omega}, R) > 0$ such that
$$ |\nabla u(x_0)| \le L.$$
\label{lem-gradint}
\end{enumerate}
\end{proposition}

First, observe that if $u$ is a classical solution of \eqref{eq-fgraph}, then $u$ satisfies
\begin{equation}
|\nabla u|^2 \Delta u + b(|\nabla u|) \nabla^2u(\nabla u, \nabla u) = -f(x,u) \, c( |\nabla u| ),
\label{mingrapheqOpen}
\end{equation} where $b(s)= -s^2/(s^2+1)$ and $c(s) = s^2\sqrt{1+s^2}$.

The next lemma is very close to Lemma 6 of \cite{RT} and therefore we omit its proof. 
\begin{lem}
If $u$ solves \eqref{mingrapheqOpen}, then in an orthonormal frame $E_1, \dots, E_n$ with $E_1=|\nabla u|^{-1}\nabla u$, the following equality holds
\begin{align*} (b+1) |\nabla u|  \nabla^2  |\nabla u|(E_1,E_1)  &+ |\nabla u| \sum_{i=2}^n  \nabla^2 |\nabla u|(E_i,E_i) \\[5pt] 
                           + b' |\nabla u| \nabla^2 u(E_1,E_1)^2 &+ b\sum_{i=2}^n \nabla^2 u (E_1,E_i)^2 \\[5pt]
                           - \sum_{i=1,j=2}^n \nabla^2 u(E_i,E_j)^2 - \Ric(\nabla u,\nabla u) &= \frac{2fc}{|\nabla u|^4} \nabla^2 u(\nabla u, \nabla u) \\[5pt]
                                                                                              &- \frac{f}{|\nabla u|^2}\langle \nabla c, \nabla u\rangle - \frac{c}{|\nabla u|^2} \langle \tilde{\nabla} f , \nabla u \rangle,  
\end{align*}
where $\tilde{\nabla}f = \nabla_x f(x,u) +  f_t(x,u) \nabla u.$
\label{lemmaAuxiliar1}
\end{lem}

As in \cite{RT} we obtain an estimate for $|\nabla u|$ by considering a function of the form \begin{equation}\label{eq-funG}
G(x):= g(x)h(u(x))F(|\nabla u(x)|)
\end{equation}
and assuming that this function attains its maximum at an interior point $y_0$ of $\Omega$.
Then the matrix $(\nabla^2 \ln G (E_i, E_j))_{i,j}$ is nonpositive at $y_0$ and it holds
$$ 0 \ge \theta:= (b+1) \nabla^2 \ln G (E_1, E_1)(y_0) + \sum_{i \ge 2} \nabla^2 \ln G (E_i, E_i)(y_0).$$
At the end, with appropriate choices of $g$, $h$ and $F$, the inequality above 
 gives an upper bound for $|\nabla u|.$ 
 
The next lemma is the version of Lemma 7 of \cite{RT} to our setting and its proof follows the same steps as the ones presented there.

\begin{lem}
If $y_0 \in \Omega$ is a local maximum of $G$ and $\nabla u(y_0) \ne 0$, then at $y_0$
\begin{equation}
\frac{F'}{F}\nabla^2u(E_1,E_i) = -\frac{1}{g}\langle \nabla g, E_i\rangle - \frac{h'}{h} \langle \nabla u, E_i \rangle \text{  for }i\in \{1, \dots, n\}
\label{criticalRelation}
\end{equation}
and 
\begin{align}
\theta &= \left[ -\frac{F'b'}{F} + (b+1) \left( \frac{F''}{F} - \frac{(F')^2}{F^2} \right) \right] \nabla^2 u(E_1,E_1)^2 \nonumber \\[5pt] 
       &+ \frac{F'}{F|\nabla u|}\sum_{i\ge 2, j=1}^n \nabla^2 u(E_i,E_j)^2 \nonumber \\[5pt] 
       &+ \left[-\frac{F'b}{F|\nabla u|} + \frac{F''}{F} - \frac{(F')^2}{F^2} \right] \sum_{i \ge 2}^n \nabla^2 u(E_1,E_i)^2 \nonumber \\[5pt] 
       &+ (b+1)\left( \frac{h''}{h} - \frac{(h')^2}{h^2}\right)|\nabla u|^2 +\frac{|\nabla u| F'}{F} \Ric(E_1,E_1) \nonumber \\[5pt] 
       &+ \frac{1}{g} \left[ (b+1) \nabla^2g (E_1,E_1) + \sum_{i \ge 2} \nabla^2 g(E_i,E_i) \right] \nonumber \\[5pt]
       &-\frac{1}{g^2} \left[ (b+1)\langle \nabla g, E_1 \rangle^2 + \sum_{i \ge 2} \langle \nabla g, E_i\rangle^2 \right] 
       +\Bigg\{   \frac{2fc}{|\nabla u|^4} \nabla^2 u (\nabla u, \nabla u) \nonumber \\[5pt] 
       &- \frac{f}{|\nabla u|^2} \langle \nabla c, \nabla u \rangle - \frac{c}{|\nabla u|^2} \langle \tilde{\nabla} f, \nabla u \rangle \Bigg\} \frac{F'}{F|\nabla u|}  + \frac{h'}{h} \left[ \frac{-fc}{|\nabla u|^2} \right] \le 0.
\label{negativityOfTheMaximum} 
\end{align}
\end{lem}

We are now in position to prove Proposition \ref{lem-gradint}.
\begin{proof}
Choose $h(t) = e^{kt}$ in \eqref{eq-funG}, where $k$ is a positive constant. 
Then, from \eqref{criticalRelation}, at the maximum point $y_0$, we get
\begin{align}
\nabla^2u(E_1,E_i) &=  -\frac{F \langle \nabla g, E_i\rangle}{F'g} - k \frac{F}{F'}  \langle \nabla u, E_i \rangle \nonumber \\[5pt]
                   &= -\frac{F\, E_i(g)}{F'g} - k \frac{F}{F'}  \langle |\nabla u| E_1, E_i \rangle = -\frac{F\, E_i(g)}{F'g} -k \frac{F}{F'}|\nabla u| \delta_{1i},
\label{criticalRelationGlobal}
\end{align}
where $\delta_{1i}$ is the Kronecker delta. 
Hence, we have
$$ \nabla^2u (\nabla u, \nabla u) =  |\nabla u|^2 \nabla^2 u ( E_1 , E_1) = -k \frac{F}{F'}|\nabla u|^3 - \frac{F \, E_1(g)}{F' g}|\nabla u|^2 $$
and
\begin{align*}
 \langle \nabla c(|\nabla u|), \nabla u \rangle &= c'(|\nabla u|)  \nabla^2 u \left(\frac{\nabla u}{|\nabla u|} , \nabla u \right)  \\[5pt]
                                                &= -c'(|\nabla u|)\left( k \frac{F}{F'}|\nabla u|^2 + \frac{F \, E_1(g)}{F' g}|\nabla u|\right) \\[5pt]
&= \! -k \frac{F}{F'} \! \left(\frac{2|\nabla u|^3 + 3 |\nabla u|^5}{\sqrt{1+|\nabla u|^2}}\right) \! - \frac{F E_1(g)}{F' g}  \! \left(\frac{2|\nabla u|^2 + 3 |\nabla u|^4}{\sqrt{1+|\nabla u|^2}}\right)\! .
\end{align*}
Therefore, we can obtain an expression for the last four terms of \eqref{negativityOfTheMaximum}:
\begin{align*}
 &\left\{ \frac{2fc}{|\nabla u|^4} \nabla^2 u (\nabla u, \nabla u) - \frac{f}{|\nabla u|^2} \langle \nabla c, \nabla u \rangle 
- \frac{c}{|\nabla u|^2} \langle \tilde{\nabla} f, \nabla u \rangle \right\}\frac{F'}{F|\nabla u|}   -   \frac{kfc}{|\nabla u|^2}  \\[5pt]
 &= -\frac{2fc}{|\nabla u|^5} \left( k \frac{F}{F'}|\nabla u|^3 + \frac{F \, E_1(g)}{F' g}|\nabla u|^2 \right)\frac{F'}{F}  - \frac{c}{|\nabla u|^3} \langle \tilde{\nabla} f, \nabla u \rangle \frac{F'}{F}  -   \frac{kfc}{|\nabla u|^2} \\[5pt]
 &+ \frac{f}{|\nabla u|^3} \left[k \frac{F}{F'} \, \left(\frac{2|\nabla u|^3 + 3 |\nabla u|^5}{\sqrt{1+|\nabla u|^2}}\right) +\frac{F\, E_1(g)}{F' g} \, \left(\frac{2|\nabla u|^2 + 3 |\nabla u|^4}{\sqrt{1+|\nabla u|^2}}\right)\right] \frac{F'}{F} \\[5pt]
 &= kf\left[ -3 \sqrt{1+|\nabla u|^2} + \frac{2 + 3 |\nabla u|^2}{\sqrt{1+|\nabla u|^2}} \right] - \frac{c}{|\nabla u|^3} \langle \tilde{\nabla} f, \nabla u \rangle \frac{F'}{F} \\[5pt]
 & -\frac{f E_1(g)}{|\nabla u| g} \left[2 \sqrt{1+|\nabla u|^2} - \frac{2 + 3 |\nabla u|^2}{\sqrt{1+|\nabla u|^2}} \right] \\[5pt]
 &= - \frac{kf}{ \sqrt{1+|\nabla u|^2} } - \frac{F'\sqrt{1+|\nabla u|^2}}{F |\nabla u|}\langle \nabla_x f(x,u) +  f_t(x,u) \nabla u, \nabla u \rangle \\[5pt]
 & + \frac{f E_1(g)}{ g} \left[ \frac{|\nabla u|}{\sqrt{1+|\nabla u|^2}} \right].
\end{align*}

\

\noindent Hence, since $f$ satisfies \eqref{boundOf-f-form} (observe that $f_t \le 0$), we have
\begin{align*}
 &\left\{ \frac{2fc}{|\nabla u|^4} \nabla^2 u (\nabla u, \nabla u) - \frac{f}{|\nabla u|^2} \langle \nabla c, \nabla u \rangle 
- \frac{c}{|\nabla u|^2} \langle \tilde{\nabla} f, \nabla u \rangle \right\}\frac{F'}{F|\nabla u|}   -   \frac{kfc}{|\nabla u|^2} \\[5pt]
&\ge - \frac{k A}{ \sqrt{1+|\nabla u|^2} } - \frac{A F' \sqrt{1+|\nabla u|^2}}{F |\nabla u|} \\[5pt] 
&\quad -  f_t(x,u) \frac{F' |\nabla u| \sqrt{1+|\nabla u|^2}}{F} - \frac{A| E_1(g)| |\nabla u|}{g \sqrt{1+|\nabla u|^2}} \\[5pt]
& \ge - \frac{k A}{ \sqrt{1+|\nabla u|^2} } - \frac{A F' \sqrt{1+|\nabla u|^2}}{F |\nabla u|} - \frac{A| E_1(g)| |\nabla u|}{g \sqrt{1+|\nabla u|^2}}.
\end{align*}
assuming that $g, F >0$ and $F' \ge 0$.

\noindent Then, inequality \eqref{negativityOfTheMaximum} implies that
\begin{align}
 0 \ge \theta &\ge \left[ -\frac{F'b'}{F} + (b+1) \left( \frac{F''}{F} - \frac{(F')^2}{F^2} \right) \right] \nabla^2 u(E_1,E_1)^2 \nonumber \\[5pt]
       &+ \frac{F'}{F|\nabla u|}\sum_{i\ge 2, j=1}^n \nabla^2 u(E_i,E_j)^2 \nonumber \\[5pt]
       &+ \left[-\frac{F'b}{F|\nabla u|} + \frac{F''}{F} - \frac{(F')^2}{F^2} \right] \sum_{i \ge 2}^n \nabla^2 u(E_1,E_i)^2 \nonumber \\[5pt]
       &+ \frac{|\nabla u| F'}{F} \Ric(E_1,E_1) + \frac{1}{g} \left[ (b+1) \nabla^2g (E_1,E_1) + \sum_{i \ge 2} \nabla^2 g(E_i,E_i) \right] \nonumber \\[5pt] 
       &-\frac{1}{g^2} \left[ (b+1)\langle \nabla g, E_1 \rangle^2 + \sum_{i \ge 2} \langle \nabla g, E_i\rangle^2 \right] - \frac{k A}{ \sqrt{1+|\nabla u|^2} } \nonumber \\[5pt]
       &- \frac{A F' \sqrt{1+|\nabla u|^2}}{F |\nabla u|} - \frac{A| E_1(g)| |\nabla u|}{g \sqrt{1+|\nabla u|^2}} .
\label{negativityOfTheMaximumNEW} 
\end{align}
Now we can prove (a) and (b).

\

\noindent Proof of (a): Choose $F(s)=s$ and $g(x)\equiv 1$. Then, from \eqref{negativityOfTheMaximumNEW}, we have
\begin{align*}
0 \ge \theta &\ge \left[ -\frac{b'}{|\nabla u|} + (b+1) \left(-\frac{1}{|\nabla u|^2} \right) \right] \nabla^2 u(E_1,E_1)^2 \\[5pt]
       &+ \frac{1}{|\nabla u|^2}\sum_{i\ge 2, j=1}^n \nabla^2 u(E_i,E_j)^2 + \left[-\frac{b}{|\nabla u|^2}  - \frac{1}{|\nabla u|^2} \right] \sum_{i \ge 2}^n \nabla^2 u(E_1,E_i)^2 \\[5pt] 
       &+ \Ric(E_1,E_1) - \frac{k A}{ \sqrt{1+|\nabla u|^2} } - \frac{A  \sqrt{1+|\nabla u|^2}}{ |\nabla u|^2}.    \\[5pt]
\end{align*}
Observe that from \eqref{criticalRelationGlobal}, it follows that $\nabla^2 u(E_1,E_i) = - k |\nabla u|^2 \delta_{1i}$.
Hence, using that $b(s)=-s^2/(1+s^2)$, we conclude that
\begin{align}
0 \ge \theta &\ge \left[ \frac{|\nabla u|^2 -1}{(1+ |\nabla u|^2)^2 |\nabla u|^2} \right] \left(-k|\nabla u|^2 \right)^2 + \Ric(E_1,E_1) \nonumber \\[5pt]
       &- \frac{k A}{ \sqrt{1+|\nabla u|^2} } - \frac{A  \sqrt{1+|\nabla u|^2}}{ |\nabla u|^2}. 
\label{negativityOfTheMaximumGlobal} 
\end{align} 
If $|\nabla u(y_0)| \ge 2$, this inequality and Young's inequality imply 
\begin{align*}
0 \ge \theta &\ge \left[ \frac{3}{16 |\nabla u|^4} \right] k^2|\nabla u|^4 + \Ric(E_1,E_1)  - k A - A \\[5pt]
             &\ge \frac{3 k^2}{16} + \Ric(E_1,E_1)  - \frac{k^2}{8}- 2 A^2 - A ,
\end{align*}
that is, 
$$k  \le 4\sqrt{2A^2 + A - \Ric(E_1,E_1)}. $$ 
Therefore, if we take $k=5 \sqrt{2A^2 + A - \Ric(E_1,E_1)}$, this inequality is not satisfied and, then, the maximum of $G$ cannot happen in some interior point $y_0$ such that $|\nabla u(y_0)| \ge 2$. Thus, either
$$ \max_{\Omega} G(y) \le G(y_0) = e^{ku(y_0)} |\nabla u(y_0)| \le 2 e^{kM}$$
or
$$ \max_{\Omega} G(y) \le \max_{\partial \Omega} G(x) = \max_{\partial \Omega} e^{ku(x)} |\nabla u(x)| \le e^{k M} \max_{\partial \Omega} |\nabla u(x)|,$$
that is,
$$ |\nabla u(y)| \le e^{k M - ku(y)} (2 + \max_{\partial \Omega} |\nabla u(x)|) \le e^{2k M}(2+ \max_{\partial \Omega} |\nabla u(x)|),$$
for any $y \in \Omega$, proving (a).

\

\noindent Proof of (b): Choose $F(s)=\ln s$ and $g(x)= 1 - r(x)^2/R^2$, where $r(x)$ is the distance from $x$ to $x_0$.
First observe that from \eqref{criticalRelationGlobal} we have, at $y_0,$
$$ \frac{ \langle \nabla g, E_i\rangle}{g}  = -\frac{F'}{F} \nabla^2u(E_1,E_i)  - k |\nabla u| \delta_{1i}. $$
Then
$$ - \frac{ \langle \nabla g, E_1\rangle^2}{g^2} \ge -2 \frac{(F')^2}{F^2} \nabla^2u(E_1,E_1)^2 - 2 k^2 |\nabla u|^2 $$
and
$$ - \frac{ \langle \nabla g, E_i\rangle^2}{g^2} = - \frac{(F')^2}{F^2} \nabla^2u(E_1,E_i)^2 \quad {\rm for} \quad i\ge 2.$$
Therefore, from inequality \eqref{negativityOfTheMaximumNEW}, we get
\begin{align*}
 0 \ge \theta &\ge \left[ -\frac{F'b'}{F} + (b+1) \left( \frac{F''}{F} - 3\frac{(F')^2}{F^2} \right) \right] \nabla^2 u(E_1,E_1)^2 \nonumber \\[5pt] 
              &- 2 k^2 |\nabla u|^2 (b+1) + \frac{F'}{F|\nabla u|}\sum_{i\ge 2, j=1}^n \nabla^2 u(E_i,E_j)^2 \nonumber \\[5pt] 
              &+ \left[-\frac{F'b}{F|\nabla u|} + \frac{F''}{F} - 2 \frac{(F')^2}{F^2} \right] \sum_{i \ge 2}^n \nabla^2 u(E_1,E_i)^2 \nonumber \\[5pt]
       &+ \frac{|\nabla u| F'}{F} \Ric(E_1,E_1) + \frac{1}{g} \left[ (b+1) \nabla^2g (E_1,E_1) + \sum_{i \ge 2} \nabla^2 g(E_i,E_i) \right] \nonumber \\[5pt] 
       & - \frac{k A}{ \sqrt{1+|\nabla u|^2} } - \frac{A F' \sqrt{1+|\nabla u|^2}}{F |\nabla u|} - \frac{A| E_1(g)| |\nabla u|}{g \sqrt{1+|\nabla u|^2}} .
\end{align*}
Since 
$$ \frac{F'}{F|\nabla u|}\sum_{i\ge 2, j=1}^n \nabla^2 u(E_i,E_j)^2 \ge \frac{F'}{F|\nabla u|}\sum_{i\ge 2}^n \nabla^2 u(E_1,E_i)^2,$$
it follows that
\begin{align}
0 \ge \theta &\ge \left[ -\frac{F'b'}{F} + (b+1) \left( \frac{F''}{F} - 3\frac{(F')^2}{F^2} \right) \right] \nabla^2 u(E_1,E_1)^2 \nonumber \\[5pt] 
       &- 2 k^2 |\nabla u|^2 (b+1) + \left[\frac{F'(1-b)}{F|\nabla u|} + \frac{F''}{F} - 2 \frac{(F')^2}{F^2} \right] \sum_{i \ge 2}^n \nabla^2 u(E_1,E_i)^2 \nonumber \\[5pt]
       &+ \frac{|\nabla u| F'}{F} \Ric(E_1,E_1) + \frac{1}{g} \left[ (b+1) \nabla^2g (E_1,E_1) + \sum_{i \ge 2} \nabla^2 g(E_i,E_i) \right] \nonumber \\[5pt] 
       & - \frac{k A}{ \sqrt{1+|\nabla u|^2} } - \frac{A F' \sqrt{1+|\nabla u|^2}}{F |\nabla u|} - \frac{A| E_1(g)| |\nabla u|}{g \sqrt{1+|\nabla u|^2}} .
\label{negativityOfTheMaximumNEWFor-b} 
\end{align}

\

\

\noindent {\bf Claim:} Choosing $k=1,$ for $\Ric^-= - \displaystyle \min_{|\eta|=1} \min \{ \Ric(\eta, \eta), 0\}$, it holds that 
\begin{align}
 g(y_0) \ln |\nabla u(y_0)|  &\le 8 \left(  2 + 2A + \frac{1+ 2A}{R} \right. \nonumber \\[5pt] 
                             &+ \left. \max_{B_R(x_0)}|\Ric^-| + \frac{n}{r^2} \max_{B_R(x_0)}| \nabla^2 r^2| \right).
                             \label{eq-claim}
\end{align}

If $|\nabla u(y_0)| < e^{12},$ the inequality follows from $g\leq 1.$

Therefore we prove the claim assuming $|\nabla u(y_0)| \ge e^{12}.$ Observe that
\begin{align*}
  -\frac{F'b'}{F} + (b+1) \left( \frac{F''}{F} - 3\frac{(F')^2}{F^2} \right)  &= \frac{(|\nabla u|^2-1)\ln |\nabla u| - 3(|\nabla u|^2 + 1)}{ |\nabla u|^2 (1+|\nabla u|^2)^2  (\ln |\nabla u|)^2}\\[5pt]
  &\ge \frac{1}{4}\,\left( \frac{(|\nabla u|^2+1)\ln |\nabla u|}{ |\nabla u|^2 (1+|\nabla u|^2)^2  (\ln |\nabla u|)^2}\right) \\[5pt]
 &\ge \frac{1}{4}\left( \frac{1}{ |\nabla u|^2 (1+|\nabla u|^2)  (\ln |\nabla u|)}\right).
\end{align*}
We have also that
\begin{align*}
 \frac{F'(1-b)}{F|\nabla u|} + \frac{F''}{F} - 2 \frac{(F')^2}{F^2} &= \frac{|\nabla u|^2 \ln |\nabla u| -2 (1+|\nabla u|^2)}{|\nabla u|^2 (1+|\nabla u|^2) (\ln |\nabla u|)^2} > 0, \\[5pt]
\end{align*}
These two inequalities and \eqref{negativityOfTheMaximumNEWFor-b} yield
\begin{align}
 0 \ge \theta &\ge \left[ \frac{1}{4 |\nabla u|^2 (1+|\nabla u|^2)  (\ln |\nabla u|)} \right] \nabla^2 u(E_1,E_1)^2 - \frac{2 k^2 |\nabla u|^2}{1+ |\nabla u|^2}\nonumber \\[5pt]
       &+ \frac{1}{\ln |\nabla u|} \Ric(E_1,E_1) + \frac{1}{g} \left[ (b+1) \nabla^2g (E_1,E_1) + \sum_{i \ge 2} \nabla^2 g(E_i,E_i) \right] \nonumber \\[5pt] 
       & - \frac{k A}{ \sqrt{1+|\nabla u|^2} } - \frac{A  \sqrt{1+|\nabla u|^2}}{ |\nabla u|^2 \ln |\nabla u|} - \frac{A| E_1(g)| |\nabla u|}{g \sqrt{1+|\nabla u|^2}} .
\label{InqualityForGradient-1} 
\end{align}
Using that 
\begin{align*}
  \nabla^2u(E_1,E_1) &= - \frac{F}{F'} \frac{ \langle \nabla g, E_1\rangle}{g} - k \frac{F}{F'} |\nabla u|   \\[5pt]
                     &= -\frac{ |\nabla u| \ln |\nabla u| E_1(g)}{g} - k  |\nabla u|^2 \ln |\nabla u|,
\end{align*}
we get
$$  \nabla^2u(E_1,E_1)^2 \ge  -\frac{2 k |\nabla u|^3 (\ln |\nabla u|)^2 |E_1(g)|}{g} + k^2  |\nabla u|^4 (\ln |\nabla u|)^2.$$
Therefore, if $|\nabla u(y_0)| \ge e^{12}$,
\begin{align}
 0 \ge \theta &\ge  \frac{k^2  |\nabla u|^2 (\ln |\nabla u|)}{4 (1+|\nabla u|^2)}  - \frac{k |\nabla u| (\ln |\nabla u|) |E_1(g)|}{2 g (1+|\nabla u|^2) } - 2 k^2 \nonumber \\[5pt]
       &+ \frac{1}{\ln |\nabla u|} \Ric(E_1,E_1) + \frac{1}{g} \left[ (b+1) \nabla^2g (E_1,E_1) + \sum_{i \ge 2} \nabla^2 g(E_i,E_i) \right] \nonumber \\[5pt] 
       & - \frac{k A}{ \sqrt{1+|\nabla u|^2} } - \frac{A  \sqrt{1+|\nabla u|^2}}{ |\nabla u|^2 \ln |\nabla u|} - \frac{A| E_1(g)| |\nabla u|}{g \sqrt{1+|\nabla u|^2}} .
\label{InqualityForGradient-2} 
\end{align}
Since $$\nabla^2 g(E_i,E_i) \ge  -\frac{| \nabla^2 r^2|}{R^2}, $$
$b+1 < 1$, $|E_1(g)| \le |\nabla g| \le \frac{2r}{R^2} \le \frac{2}{R}$ and $0 < g \le 1$, we conclude that
$$ 0 \ge  \frac{k^2 (\ln |\nabla u|)}{8}  - \frac{k   }{R g} - \frac{2 k^2}{g} - \frac{|\Ric^-|}{g} - \frac{n| \nabla^2 r^2|}{R^2g}    -\frac{k A}{g} - \frac{A}{g} - \frac{2A}{Rg}. $$
Hence, for $k=1$, we obtain \eqref{eq-claim}.
$$g(y_0) \ln |\nabla u(y_0)| \le 8 \left(  2 + 2A + \frac{1+ 2A}{R} + \max_{B_R(x_0)}|\Ric^-| + \frac{n}{R^2} \max_{B_R(x_0)}| \nabla^2 r^2| \right)$$

Therefore,
\begin{align*}
 e^{u(x_0)} &\ln |\nabla u(x_0)| = G(x_0) \le G(y_0) \le g(y_0)e^{u(y_0)} \ln |\nabla u(y_0)| \\[5pt]
                                &\le 8 e^M \left(  2 + 2A + \frac{1+ 2A}{R} + \max_{B_R(x_0)}|\Ric^-| + \frac{n}{R^2} \max_{B_R(x_0)}| \nabla^2 r^2| \right),
\end{align*}
that is,
\begin{align}
 |\nabla u(x_0)| &\le \exp\left[8 e^{2M} \left(  2 + 2A + \frac{1+ 2A}{R} \right. \right. \nonumber \\[5pt] 
                 &+ \left. \left. \max_{B_R(x_0)}|\Ric^-| + \frac{n}{R^2} \max_{B_R(x_0)}| \nabla^2 r^2| \right)\right].  
\label{IneqFinalForGrad}
\end{align}

Since $\max_{B_R(x_0)}| \nabla^2 r^2|$ is bounded by a constant depending on the curvature and on $R,$ the result follows.
\end{proof}

\

We are now in position to prove Theorem \ref{thm-existence}.

\begin{proof}
[Proof of Theorem \ref{thm-existence}]
We begin by assuming that $\varphi\in
C^{2,\alpha}\left(\overline{\Omega}\right).$ Consider the following family
of Dirichlet problems
\begin{equation}
\begin{cases}
\dv\dfrac{\nabla v}{\sqrt{1+|\nabla v|^{2}}}+\tau f(x,v)=0\text{ in }\Omega,\\
v=\tau\varphi\text{ in }\partial\Omega,\ 0\leq\tau\leq1.
\end{cases}
\label{eqfamily}%
\end{equation}

Observe that from Lemma \ref{lem-altEgradbordo}, any solution $v_{\tau}$ to \eqref{eqfamily} is bounded by a constant that does not depend on $\tau$.  So Proposition \ref{lem-gradint} applies. Hence, there exists a constant $C,$ not depending on $\tau,$ such that for any solution $v_{\tau}$ to \eqref{eqfamily},
\[
\lVert v_{\tau}\rVert_{C^{1}(\overline{\Omega})}\leq C.
\]
Thanks to this estimate, we obtain a solution $v\in C^{2,\alpha}\left(
\overline{\Omega}\right)  $ to \eqref{eqfamily} by using the Leray-Schauder
method \cite[Theorem 13.8]{GilTru}.

If $\varphi\in C^{0}\left(  \partial
\Omega\right)  $ we take an approximating sequence of $\varphi$ by
$C^{2,\alpha}$ functions. Using then the previous case, the comparison
principle, Lemma \ref{lem-altEgradbordo} and Proposition \ref{lem-gradint} we obtain the existence of a solution
$v\in C^{2}\left(  \Omega\right)  \cap C^{0}\left(  \overline{\Omega}\right)
$ to \eqref{eq-fgraph}. 
This concludes the proof of Theorem \ref{thm-existence}.  
\end{proof}

\section{Scherk type solutions in the hyperbolic space} \label{sec-STSolutions}

From now on we concentrate in the hyperbolic space $\Hn.$ In order to prove Theorems \ref{thm-Scherk} and \ref{ContinuityOfSolutionWithSing}, we construct barriers that take value $+\infty$ in a totally  geodesic hypersphere of $\Hn.$

\begin{defin}
Let $S$ be a totally geodesic hypersphere (or a geodesic if $n=2$) of
$\mathbb{H}^{n}$ and $B$ be one connected component of $\mathbb{H}^{n}\backslash S$. 
Given $f\in C^{1}(\Hn\times\mathbb{R})$ and a constant $c$, if $u\in C^2(B)\cap C^0(\overline{B})$ is a solution of
\begin{equation}
\begin{cases}
\label{ScherkProblem} 
Q(v)= f(x,v)\text{ in } B\\
v = c\text{ on }\partial_{\infty}B\\
v= +\infty\text{ on } S,
\end{cases}
\end{equation}
where $\partial_{\infty}B$ is the asymptotic boundary of $B,$ we call $u$ a Scherk type solution to problem \eqref{ScherkProblem}. \label{def-Scherkprob}

Analogously we define Scherk type sub and supersolutions.
\end{defin}

Our next result is about the existence of Scherk type solutions and for that we assume that $f$ satisfies:
\begin{enumerate}[label=(\text{\Roman*}), ref=\text{\Roman*}]
\item \label{propphi1} for some fixed $o\in\mathbb{H}^{n}$, there exists a continuous
decreasing function $\phi:[0,+\infty)\rightarrow\mathbb{R}$ such that
\[
|f(x,t)|\leq\phi(r(x)),
\]
where $r(x)=\mathrm{dist}(x,o)$, 
and $\displaystyle\int_{0}^{+\infty}\sqrt{\phi(r)}\,dr<\infty$; 

\item \label{propphi2} there is a continuous function $h:\mathbb{R}\rightarrow\mathbb{R},$ such that $h(t)\rightarrow0$ as $t\rightarrow\pm\infty,$ for which
\[
|f(x,t)|\leq h(t), \, \forall x\in \Hn.
\]
\end{enumerate}

We also call $B$ from definition \ref{def-Scherkprob} a totally geodesic hyperball and we denote by $d=d_S$ the signed distance function to $S,$ positive in $B.$ 

Conditions \ref{propphi1} and \ref{propphi2} guarantee the existence of a nice function $\psi=\psi_S$ as stated below:

\begin{proposition}\label{prop-Psi}
Given $f\in C^{1}(\Hn\times\mathbb{R})$ satisfying conditions \ref{propphi1} and \ref{propphi2}, there exists a nonnegative $C^{1}$ function $\psi:\mathbb{R}\times\mathbb{R}\rightarrow\mathbb{R}$ such that 
\begin{enumerate}[label=(\roman{enumi})]
 \item \label{psiDefinitiona}  $\psi
(d(x),t)\geq|f(x,t)|$, for any $(x,t)\in
\mathbb{H}^{n}\times\mathbb{R}$;
\item  \label{psiDefinitionb}  $\psi(d,t)$, $\frac{\partial\psi
}{\partial d}(d,t)$ and $\frac{\partial\psi}{\partial t}(d,t)$ are bounded
functions; 
\item   \label{psiDefinitionc} for $d\in\mathbb{R}$, the map $t\mapsto\psi(d,t)$  is decreasing for $t\geq0$ and constant for $t\leq 0$; 
\item  \label{psiDefinitiond} for $t\in\mathbb{R}$, the map $d\mapsto\psi(d,t)$ is increasing in $(-\infty,\tilde{d})$  and 
decreasing in $(\tilde{d},+\infty)$, where $\tilde{d}$ is some real  number that does not depend on $t$; 
\item \label{psiDefinitionef} $\psi(d,t)$ converges to zero uniformly in
$t\in\mathbb{R}$ as $d\rightarrow\pm\infty$ and uniformly in $d\in\mathbb{R}$ as $t\rightarrow+\infty$; 
\item \label{psiDefinitiong} for
any $t\in\mathbb{R}$, $\displaystyle\int_{0}^{+\infty}\psi(s,t)\;ds<+\infty$. 
\end{enumerate}
\end{proposition}

\begin{proof} 
First, observe that we can assume w.l.g. that $\phi$ and $h$ are $C^1$ functions, $\phi'(0)=0$, $h$ is even and decreasing on $[0,+\infty)$.
The proof follows by considering 
$$\psi(d,t):= \left\{ \begin{array}{rr}  \sqrt{ \phi(|d - d(o)|) h(t)} & {\rm if} \quad t \ge 0 \\[5pt]
                                         \sqrt{ \phi(|d - d(o)|) h(0)} & {\rm if} \quad t < 0.
\end{array} \right.                                     
$$
The conditions (ii)-(vi) can be verified directly from the definition. To prove condition $(i)$, 
let $x \in \mathbb{H}^n$. By the triangle inequality 
$$ r(x) + d(x) \ge d(o) \text{ and }
r(x) + d(o) \ge d(x),$$
that is, $r(x)\ge |d(x) - d(o)|$. (This holds even if $d(x) < 0$ or $d(o) < 0$.) Hence, using that $\phi$ is decreasing and hypothesis (1), we have that
$$ |f(x,t)| \le \phi(r(x))\le \phi(|d(x) - d(o)|).$$
From this and \ref{propphi2}, we have $|f(x,t)|^2\le \phi(|d(x) - d(o)|) h(t) \le \psi^2(d(x),t)$.
\end{proof}

The main result of this section is the following:

\begin{thm}
\label{s}
 $\mathbb{H}^{n}\backslash S$. 
Let $B\subset \Hn$ be a totally  geodesic hyperball. Suppose that $f\in C^{1}(\Hn\times\mathbb{R})$ is nonnegative and satisfies $f_t(x,t)\le 0$  in $\Hn\times\mathbb{R},$ \ref{propphi1} and \ref{propphi2}. Then, for any constant $c\in \mathbb{R}$, there exists a solution $u$ to
the problem \eqref{ScherkProblem}.
Besides, if $f$ is not necessarily nonnegative and satisfies only conditions \ref{propphi1} and \ref{propphi2}, this Dirichlet problem has a supersolution. If we replace $v = +\infty$ on $S$ by $v = -\infty$ on $S$ and assume \ref{propphi1} and \ref{propphi2}, the problem has a subsolution. 
\end{thm}
To prove this theorem, our main task is to construct supersolutions for this equation. The
existence of Scherk type solutions will then be an immediate consequence of Perron's method, which due to Theorem \ref{thm-existence} applies in our setting.

Now let us explain our strategy to construct supersolutions for \eqref{ScherkProblem}. We will look for a solution $w=w(d(x))$, where $d(x)=dist(x,S)$, $S=\partial B$,  to the following problem 
\begin{equation}
\begin{cases}
Q(v\circ d)= \psi(d ,v \circ d) \text{ in } B\\
v = c \text{ on } \partial_{\infty} B\\
v = + \infty\text{ on }S.
\end{cases}  \label{superScherkEDP}%
\end{equation}

Since $\Delta d= (n-1)\tanh d$, we can rewrite \eqref{superScherkEDP} as 
\begin{equation}
\begin{cases}
\displaystyle \frac{w^{\prime\prime}}{(1 + w^{\prime2})^{3/2}} + (n-1)
\tanh(d) \left(\frac{w^{\prime}}{\sqrt{1 + w^{\prime2}}}\right) = -\psi(d,w)
\text{ for }  d > 0\\
w(+\infty) = c\\
w(0) = + \infty. \label{superScherkEDO}
\end{cases}  %
\end{equation}

\

Next, we set $g=\frac{ w^{\prime}}{\sqrt{1+ w^{\prime2}}}$ on $[0,+\infty)$.
Observe that $-1 < g < 1$ and that the ODE in \eqref{superScherkEDO} rewrites as the following system for $(w,g)$
\begin{equation}
\begin{cases}
w^{\prime}(d) = \displaystyle \frac{g(d)}{\sqrt{1-g^{2}(d)}}\\[10pt]
g^{\prime}(d) = -(n-1)\tanh(d) \; g(d) - \psi(d, w(d)).
\end{cases}  \label{systemForWG}%
\end{equation}
or
\[
\left[
\begin{array}
[c]{c}%
w(d)\\
g(d)
\end{array}
\right]  ^{\prime}= F(d,w(d),g(d)),
\]
where $F: \mathbb{R} \times\mathbb{R} \times(-1,1) \to\mathbb{R}^{2}$
corresponds to the right-hand side of \eqref{systemForWG}. Given $d_{0} > 0$,
$h \in\mathbb{R}$ and $\gamma\in(-1,1)$, we consider the initial condition
\begin{equation}
\begin{cases}
w(d_{0}) = h\\
g(d_{0}) = \gamma.
\end{cases}  \label{InitialConditionsForWG}%
\end{equation}
Note that, from \eqref{systemForWG} and \eqref{InitialConditionsForWG}, we
get
\begin{equation}
w_{\gamma}(d) = h + \int_{d_{0}}^{d} \frac{g_{\gamma}(t)}{\sqrt{1-g_{\gamma
}^{2}(t)}}\; dt \label{wExpressionByG}%
\end{equation}
and
\begin{equation}
g_{\gamma}(d) = \frac{1}{(\cosh d)^{n-1}} \left(  \gamma\cosh^{n-1}(d_{0}) -
\int_{d_{0}}^{d} \psi(s,w_{\gamma}(s)) (\cosh s)^{n-1} ds \right)  .
\label{gExpressionByW}%
\end{equation}
Since $F$ is $C^{1}$, from the classical theory, there exists only one maximal
solution $(w_{d_{0},h,\gamma},g_{d_{0},h,\gamma})$ to the system
\eqref{systemForWG} with initial condition \eqref{InitialConditionsForWG}.
Let $I_{d_{0},h,\gamma}$ be the domain of this solution. Observe that
$w_{d_{0},h,\gamma}$ is the solution of the second order ODE in
\eqref{superScherkEDO} with the initial conditions
\[
w(d_{0}) = h \quad\mathrm{and} \quad w^{\prime}(d_{0}) = \frac{\gamma}%
{\sqrt{1-\gamma^{2}}}.
\]
To solve \eqref{superScherkEDO}, we have to prove that there exist $d_{0} >
0$, $h \in\mathbb{R}$ and $\gamma\in(-1,1)$ such that $w_{d_{0},h,\gamma
}(0)=+\infty$ and $w_{d_{0},h,\gamma}(+\infty)=c$. 

To do so, we first fix $h$ and $d_{0} >0$, and then study the behavior of
$I_{\gamma}:=I_{d_{0},h,\gamma}$, $w_{\gamma}:=w_{d_{0},h,\gamma}$ and
$g_{\gamma}:=g_{d_{0},h,\gamma}$ as $\gamma$ varies. We will prove, in Proposition \ref{uniquenessOfGamma-0}, that there exists a unique $\gamma_0=\gamma_0(d_0,h)$ such that $I_{\gamma_0} = (0,\infty)$. In a second moment, we will show in Lemma \ref{lemsurjective} that the application $h\mapsto \lim_{d\to \infty} w_{d_0,h,\gamma_0(d_0,h)}(d)$ is well-defined and surjective on $\R.$ This will establish the existence of solution for the problem \eqref{superScherkEDP}.

Let us now put into practice the strategy described above. 

\ 

First, we establish several properties of $g_{\gamma}$ that we will use extensively along the proof. We begin by proving some bounds for $g^\prime_\gamma$ and $g^{\prime \prime}_\gamma$.

\begin{lem}
Let $d_{0} >0$ and $\gamma\in(-1,1)$. Then we have 
\begin{enumerate}
\item[-]  $|g^{\prime}_{\gamma}| \le n-1 + $  $\max \psi$; 
\item[-] $g^{\prime\prime}_{\gamma}$ is bounded from
above (resp. from below) in the set $$\{ d \in I_{\gamma} \; | \; w^{\prime}_{\gamma}(d) \le0\  (\text{resp. }\ge0) \};$$
\item[-] if $g_{\gamma}(d_{1}) > 0$ for some $d_{1} \in I_{\gamma}$, then $g_{\gamma
}(d) > 0$ for any $d \le d_{1}$, $d\in I_{\gamma}$. (If $g_{\gamma}(d_{1}) \ge 0$ for some $d_{1} \in I_{\gamma}$, then $g_{\gamma
}(d) \ge 0$ for any $d \le d_{1}$, $d\in I_{\gamma}$.)
\end{enumerate}
\label{boundnessOfTheDerivativesOfg}
\end{lem}

\begin{proof}
Remind that $-1 <g_{\gamma}(d) < 1$ for any $d\in
I_{\gamma}$ and 
$\psi$ is bounded from \ref{psiDefinitionc}, \ref{psiDefinitiond} and \ref{psiDefinitiond}. Then the bound on $|g_{\gamma}^{\prime}|$ follows directly from \eqref{systemForWG}.

Next, we prove that $g^{\prime\prime}_{\gamma}$ is bounded from above in the set $\{
d \in I_{\gamma} \; | \; w^{\prime}_{\gamma}(d) \le0 \}$ (the other statement follows in the same way). Observe that,
differentiating the second equation of \eqref{systemForWG}, we obtain%

\begin{align*}
g_{\gamma}^{\prime\prime}(d)  &  =-(n-1)\;\mathrm{sech}^{2}(d)\,g_{\gamma
}(d)-(n-1)\tanh(d)\,g_{\gamma}^{\prime}(d)\\[5pt]
&  \quad-\frac{\partial\psi}{\partial d}(d,w_{\gamma}(d))-\frac{\partial\psi
}{\partial t}(d,w_{\gamma}(d))\,w_{\gamma}^{\prime}(d).
\end{align*}
Using that $w_{\gamma}^{\prime}(d)\leq0$ and $\dfrac{\partial \psi}{\partial t}\leq0$ (from \ref{psiDefinitionc}), we get 
\[
g_{\gamma}^{\prime\prime}(d)\leq-(n-1)\;\mathrm{sech}^{2}(d)\,g_{\gamma
}(d)-(n-1)\tanh(d)\,g_{\gamma}^{\prime}(d)-\frac{\partial\psi}{\partial
d}(d,w_{\gamma}(d)).
\]
Since $|g_{\gamma}^{\prime}|$ and $\dfrac{\partial \psi}{\partial d}$ are bounded (see \ref{psiDefinitionb}), it follows that $g_{\gamma
}^{\prime\prime}$ is bounded from above.

To prove the last statement, observe that, multiplying \eqref{systemForWG} by
$\cosh^{n-1} (d)$, we obtain
\begin{equation}
( (\cosh d)^{n-1} g_{\gamma}(d))^{\prime}=-(\cosh d)^{n-1} \psi(d,w_{\gamma}(d)) \le0,
\label{detivativeOfProductGammaCosh}%
\end{equation}
that is, $(\cosh d)^{n-1} g_{\gamma}(d)$ is non increasing. Hence, if
$g_{\gamma}(d_{1}) > 0$, then
\[
(\cosh d)^{n-1} g_{\gamma}(d) \ge(\cosh d_{1})^{n-1} g_{\gamma}(d_{1}) >0
\quad\mathrm{for} \quad d \le d_{1}, \,d\in I_{\gamma}.
\]
\end{proof}

\begin{lem}
Let $d_{0} > \tilde{d} > 0$ ($\tilde{d}$ is defined in \ref{psiDefinitiond}) be such that
\begin{equation}\frac{1}{2} (n-1) \tanh d_{0} -
\psi(d_{0},0) > 0. \label{d0Condition1}%
\end{equation}
Then, for any $\gamma\in(-1,1)$, for any $ d >
d_{0}$ ($d \in I_{\gamma}$), there holds
\begin{equation}
\min\{-\frac{1}{2} , \gamma\}< g_{\gamma}(d) \le \max\{0,\gamma\}
\label{gIsBoundedByAboveForDLarge1}
\end{equation}
Moreover, if $\gamma\ge0$, we have
\begin{equation}
\label{gIsBoundedByAboveForDLarge2}
g_{\gamma}(d) \ge0 \text{ for }  d \le d_{0}\, (d \in I_{\gamma
}).
\end{equation}

 \label{gIsBoundedByBelowForDLarge}
\end{lem}

\begin{proof}
First observe that, since $\psi(d,0) \to0$ and $\tanh d \to1$ as $d
\to+\infty$, one can find $d_{0} >0$ satisfying \eqref{d0Condition1}.
Let $a_{\gamma}:[d_{0},+\infty)\rightarrow \R$ be defined by
\[
a_{\gamma}(d) = \frac{1}{(\cosh d)^{n-1}} \left(  \gamma\cosh^{n-1}(d_{0}) -
\int_{d_{0}}^{d} \psi(s,0) (\cosh s)^{n-1} ds \right)  .
\]
 Using that  $\psi(s,0) \ge\psi(s,w_{\gamma
}(s))$, for any $s$, (which follows from \ref{psiDefinitionc}) and \eqref{gExpressionByW}, we get that
\begin{equation}
g_{\gamma}(d) \ge a_{\tilde{\gamma}}(d) \quad\mathrm{for} \quad d \ge d_{0},
\label{relationBetweenGandA}%
\end{equation}
where
$\tilde{\gamma} = \min\{ -\frac{1}{2}, \gamma\}$. 
Now we prove that $a_{\tilde{\gamma}}$ is nondecreasing. First notice that
\begin{equation}
a_{\tilde{\gamma}}^{\prime}(d) = -\psi(d,0) -(n-1)\tanh(d) \; a_{\tilde
{\gamma}}(d). \label{EquationFor-a-gamma-tilde}%
\end{equation}
Hence, from $-\tilde{\gamma} \ge1/2$ and \eqref{d0Condition1}, we deduce that
\begin{align}
a_{\tilde{\gamma}}^{\prime}(d_{0})  &  = -\psi(d_{0},0) -(n-1)\tanh(d_{0}) \;
\tilde{\gamma}\nonumber\\[2pt]
&  \ge-\psi(d_{0},0) + \frac{1}{2} (n-1)\tanh(d_{0}) > 0.
\label{a-gammaTildeLowerBound}%
\end{align}
Suppose by contradiction that there exists some $d_{1} > d_{0}$ such that $a_{\tilde{\gamma}%
}^{\prime}(d_{1}) < 0$. Since $a_{\tilde{\gamma}}^{\prime}$ is continuous and
$a_{\tilde{\gamma}}^{\prime}(d_{0}) > 0$, there is some $d_{2} \in(d_{0}%
,d_{1})$ such that $a_{\tilde{\gamma}}^{\prime}(d_{2}) =0 $ and $a_{\tilde
{\gamma}}^{\prime}(d) < 0$ for $d \in(d_{2},d_{1}]$. By the Mean Value
Theorem, there is some $d_{3} \in(d_{2},d_{1})$ such that $a_{\tilde{\gamma}%
}^{\prime\prime}(d_{3}) < 0$. Then, from \eqref{EquationFor-a-gamma-tilde}, we get that
\[
0> a_{\tilde{\gamma}}^{\prime\prime}(d_{3}) = -\psi^{\prime}(d_{3},0)
-(n-1)(\mathrm{sech}\, (d_{3}))^{2} \; a_{\tilde{\gamma}}(d_{3}) - (n-1)
\tanh(d_{3}) \; a_{\tilde{\gamma}}^{\prime}(d_{3}).
\]
We get a contradiction from the fact that $\psi(d,0)$ is decreasing in ($\tilde{d},+\infty)$, $a_{\tilde{\gamma}}$ is negative, and $a_{\tilde{\gamma}}^{\prime}(d_{3}) < 0$. Therefore we have proved that $a_{\tilde{\gamma}}(d)$ is nondecreasing for $d>d_0$. Since $a_{\tilde{\gamma}%
}^{\prime}(d_{0}) >0$,  using \eqref{relationBetweenGandA}, we conclude that 
\[
g_{\gamma}(d) \ge a_{\tilde{\gamma}}(d) > a_{\tilde{\gamma}}(d_{0})= \tilde{\gamma} \quad\mathrm{for }
\quad d > d_{0}.
\]

Finally the upper bound of \eqref{gIsBoundedByAboveForDLarge1} and \eqref{gIsBoundedByAboveForDLarge2} are direct consequences of \eqref{gExpressionByW} and
$\psi\ge0$, and they hold for any $d_0>0$ i.e. not necessarily satisfying \eqref{d0Condition1}.

\end{proof}

\begin{rem}
Let $d_0$ be such that \eqref{d0Condition1} holds and $-1 < \gamma
\le-\frac{1}{2}$. Noticing that $\tilde \gamma =\gamma$ and $a_{\tilde \gamma}(d_0)=\gamma = g_{\gamma}(d_0)$, using \eqref{relationBetweenGandA} and \eqref{a-gammaTildeLowerBound}, we obtain
\begin{equation}
\label{lowerBoundForGgammaPrimeAtD0}
g_{\gamma}^{\prime}(d_{0})\ge a^\prime_{\tilde \gamma}(d_0) \ge\frac{1}{2} (n-1) \tanh d_{0} - \psi(d_{0},0).
\end{equation}

\end{rem}

\

Thanks to the two previous lemmas, we are able to prove that the domain of our maximal solution is of the form $(d_\gamma ,+\infty)$, for some $d_\gamma < d_0$. Moreover, we charaterize the behavior of our solution when $d$ goes to $d_\gamma$. 
\begin{cor}\label{LimitOfg-gamma}
Let $d_{0} >0$ as in Lemma \ref{gIsBoundedByBelowForDLarge} and let $\gamma
\in(-1,1)$. Then the maximal interval $I_\gamma$ of $(w_\gamma, g_\gamma)$ has the form $I_{\gamma}=(d_{\gamma},
+\infty)$, where $d_{\gamma} \in[-\infty,d_{0})$. Furthermore:
\begin{enumerate}[label=(\alph{enumi})]
\item \label{cor-a}$\displaystyle \lim_{d\to+\infty}g_{\gamma}(d) = 0;$ 
\item \label{cor-b}if $d_{\gamma} \ne-\infty$, then $\displaystyle \lim_{d\to d_{\gamma}}g_{\gamma
}(d) = 1$ or $\displaystyle \lim_{d\to d_{\gamma}}g_{\gamma}(d) = -1$;
\item \label{cor-c}if $d_{\gamma} \ne-\infty$ and $\gamma\ge0$, then
$\displaystyle \lim_{d\to d_{\gamma}}g_{\gamma}(d) = 1$; 
\item \label{cor-d}if
$\displaystyle \lim_{d\to d_{\gamma}}g_{\gamma}(d) = -1$, then $g_{\gamma} <
0$ in $I_{\gamma}$.
\end{enumerate}
 
\end{cor}

\ 

\begin{proof} 
Suppose by contradiction that $I_{\gamma}=(d_{\gamma}, b)$,
where $b < +\infty$. Thanks to Lemma \ref{boundnessOfTheDerivativesOfg} and Lemma \ref{gIsBoundedByBelowForDLarge}, there exists $(\bar{w}, \bar{g})\in \R \times (-1,1)$ such that
\[
\lim_{d \to b} (d,w_{\gamma}(d), g_{\gamma}(d)) = (b, \bar{w}, \bar{g})
\in\Omega,
\]
where $\Omega= \mathbb{R} \times\mathbb{R} \times(-1,1)$ is the domain of $F$.
However by classical ODE theory, this contradicts the fact that $I_{\gamma}$
is maximal proving that $I_{\gamma}=(d_{\gamma},+\infty)$.

If $d_{\gamma} \ne-\infty$, using the same argument as before, it is clear that
$$\lim_{d \to d_{\gamma}} g_{\gamma}(d)= \pm 1$$ 
proving \ref{cor-b}. 
The proof of $\ref{cor-c}$ is a consequence of \ref{cor-b} and \eqref{gIsBoundedByAboveForDLarge2} whereas
 $\ref{cor-d}$ is a direct consequence of the last statement of Lemma
\ref{boundnessOfTheDerivativesOfg}.

Finally we prove $\ref{cor-a}$. Using \eqref{gExpressionByW}, we get, for $d\geq d_0$,
\begin{align*}
|g_{\gamma}(d)|& \le\frac{|\gamma|\cosh^{n-1}(d_{0})}{\cosh^{n-1}(d)} +
\frac{\displaystyle \left|  \int_{d_{0}}^{d} \psi(s,w_{\gamma}(s)) \cosh
^{n-1}(s) ds \; \right|  }{\cosh^{n-1}(d)}\\
&\leq \frac{ \cosh^{n-1}(d_{0})}{\cosh^{n-1}(d)} + \rho(d), 
\end{align*}
where
\begin{equation}\label{eq-defrho}
\rho(d) = \frac{ \displaystyle \int_{d_{0}}^{d} \psi(s,0) \cosh^{n-1}(s) \; ds
}{ \cosh^{n-1}(d)}.
\end{equation}
It is easy to see that the proof boils down to show that $\rho (d)\rightarrow 0$ as  $d\rightarrow 0$. Let us prove this last point.\\

 Let $\varepsilon> 0$. Since $\psi(d,0) \to0$ as $d
\to+\infty$, there exists $d_{1} > d_{0}$ such that $\psi(d,0) <
\varepsilon/2^{n}$ for $d \ge d_{1}$ and there exists $d_2>d_1$ such that, for all $d\geq d_2$, 
\[
\frac{\displaystyle   \int_{d_{0}}^{d_{1}} \psi(s,0) \cosh^{n-1}(s) ds }%
{\cosh^{n-1}(d)} \leq \dfrac{\varepsilon}{2}.
\]
Therefore, for $d \ge d_{2}$, we have
\begin{align*}
\rho(d)  &  = \frac{\displaystyle \int_{d_{0}}^{d_{1}} \psi(s,0) \cosh
^{n-1}(s) ds }{\cosh^{n-1}(d)} + \frac{\displaystyle \int_{d_{1}}^{d}
\psi(s,0) \cosh^{n-1}(s) ds }{\cosh^{n-1}(d)}\\[10pt]
&  \le\frac{\varepsilon}{2} + \frac{\displaystyle \int_{d_{1}}^{d}
\varepsilon/2^{n} \cosh^{n-1}(s) ds }{\cosh^{n-1}(d)}<\varepsilon ,
\end{align*}
proving that $\displaystyle \lim_{d\to+\infty}\rho(d)= 0$. This concludes the proof.
 \end{proof}

\begin{rem}\label{rmk-rhoint}
In the following, we will need some more refined estimate on $\rho$. More precisely, one can show that  $|\rho|$ is integrable in
$[d_{0},+\infty)$ and
\begin{equation}
\displaystyle \int_{d_{1}}^{d_{2}} \rho(t) \; dt \leq \frac{2^{n-1}}{n-1} \left(
\rho(d_{1}) + \int_{d_{1}}^{d_{2}} \psi(s,0) \, ds \right)  \quad\mathrm{for}
\quad d_{2} > d_{1} \ge d_{0}. \label{boundForIntegralOfRho}%
\end{equation}
Indeed, let $d_{2} > d_{1} \ge d_{0}$. Using the
definition of $\rho$, we have
\begin{align*}
\int_{d_{1}}^{d_{2}} \rho(t) \, dt
&  = \int_{d_{0}}^{d_{1}} \psi(s,0) \cosh^{n-1}(s) \int_{d_{1}}^{d_{2}}
\frac{1}{\cosh^{n-1}(t)} \, dt \,ds\nonumber\\[5pt]
&  \quad+ \int_{d_{1}}^{d_{2}} \psi(s,0) \cosh^{n-1}(s) \int_{s}^{d_{2}}
\frac{1}{\cosh^{n-1}(t)} \, dt \,ds\nonumber\\[5pt]
& <  \frac{2^{n-1}}{n-1} \left(
\rho(d_{1}) + \int_{d_{1}}^{d_{2}} \psi(s,0) \, ds \right)  \quad\mathrm{for}
\quad d_{2} > d_{1} \ge d_{0}, \label{integralOfRho1}%
\end{align*}
proving \eqref{boundForIntegralOfRho}. The integrability of $|\rho|$ in $[d_0, \infty )$ is then a direct consequence of \ref{psiDefinitiong}.
\end{rem}

\

For $d_{0} >0$ as in Lemma \ref{gIsBoundedByBelowForDLarge}, we set
\[
A = \{ \gamma\in(-1,1) \; \; | \; \displaystyle \inf_{d > 0 , d \in I_{\gamma
}} g_{\gamma}(d) > -1 \}.
\]
Observe that $A$ is nonempty since $0 \in A$ due to the fact $\inf_{d \in I_{\gamma}} g_{0}(d) \ge-1/2$
according to Lemma \ref{gIsBoundedByBelowForDLarge}. We define
\begin{equation}
\gamma_{0} = \inf A . \label{gammaoDefinition}%
\end{equation}

We will show in the following that $\gamma_0$ is the unique initial data such that $I_{\gamma_0}=\R^+$. Before proceeding, let us show some preliminary properties of the set $A$ and of $\gamma_0$.

\begin{lem}
$\gamma\in A$ if and only if $d_{\gamma} < 0$ or $\displaystyle{\lim_{d \to d_{\gamma}%
}g_{\gamma}(d) =1}$. \label{characterizationOFA}
\end{lem}

\begin{proof} Suppose that $\gamma\in A$. If $d_{\gamma} \ge0$, the fact that $\gamma \in A$ and {\it (b)} of Corollary \ref{LimitOfg-gamma} imply directly that $\lim_{d \to d_{\gamma}%
}g_{\gamma}(d) = 1$.

Next, we prove the reverse implication. Using that
$g_{\gamma}$ is continuous, $g_{\gamma}(d) > -1$ for any $d \in I_{\gamma}$
and, using Corollary \ref{LimitOfg-gamma}, $\lim_{d \to+\infty}g_{\gamma}(d) = 0$, we conclude that
\[
\inf_{d > 0} g_{\gamma}(d) > -1,\ \text{if } d_\gamma <0,
\]
or
\[
\inf_{d \in I_{\gamma}} g_{\gamma}(d) > -1,\ \text{if } \lim_{d \to d_{\gamma_{0}}}
g_{\gamma_{0}}(d) = 1,
\]
proving that $\gamma\in A$. 
\end{proof}

\begin{prop}
There exists $\delta>0$, that depends only on
$d_{0}$ and $\psi$, such that $\gamma_{0} \ge-1+\delta$.
\label{gamma-zero-biggerThan-1}
\end{prop}

\begin{proof}
First, using \eqref{lowerBoundForGgammaPrimeAtD0}, observe that there exists a constant $L>0$ such that $g_{\gamma}^{\prime}(d_{0}) \ge L$
for $\gamma\in(-1,-\frac{1}{2}]$. 
Furthermore, from Lemma \ref{boundnessOfTheDerivativesOfg}, there exists $M >L
/ d_{0}$ such that $g^{\prime\prime}_{\gamma}(d) \le M$ for any $d \in
I_{\gamma}$ satisfying $w_{\gamma}^{\prime}(d) \le0$. Next, we set
\[
\delta= \min\left\{  \frac{1}{2}, \frac{L^{2}}{2M} \right\}  .
\]
We will show that if $\gamma\in (-1, -1 + \delta )$, then $\gamma\not \in A$
proving that $\gamma_{0} > -1$. 
 First, we prove that, for $\gamma
\in(-1,-1+\delta)$,
\begin{equation}
g_{\gamma}^{\prime}(d) \ge0 \quad\mathrm{for} \quad d \in I_{\gamma}
\cap[d_{0}-L/M, d_{0}]. \label{g-gamma-primeIsNonnegativeInSomeInterval}%
\end{equation}
Suppose by contradiction that $g_{\gamma}^{\prime}(d_{1}) < 0$ for some $d_{1} \in
I_{\gamma} \cap[d_{0}-L/M, d_{0}]$. Noticing that $g_{\gamma}^{\prime}(d_{0}) \ge L >0$, since $\gamma \leq -1/2$, and using the continuity of $g_{\gamma}^{\prime}$, there
exists $d_{2} \in(d_{1},d_{0})$ such that $g_{\gamma}^{\prime}(d_{2}) = 0$ and
$g_{\gamma}^{\prime}> 0$ in $(d_{2},d_{0}]$. This implies that $g_{\gamma}(d)\le
g_{\gamma}(d_{0})=\gamma< 0$ for $d \in[d_{2},d_{0}]$ and, from
\eqref{systemForWG},
\[
w_{\gamma}^{\prime}(d) = \frac{g_{\gamma}}{\sqrt{1- g_{\gamma}^{2}}} \le0
\text{ for } \quad d \in[d_{2},d_{0}].
\]
Thus, we deduce from the definition of $M$ that $g_{\gamma}^{\prime\prime}(d) \le M$ for $d \in[d_{2},d_{0}]$. However, using the Mean Value Theorem, we
get, for some $\xi\in(d_{2},d_{0})$,
\[
L \le g_{\gamma}^{\prime}(d_{0}) - g_{\gamma}^{\prime}(d_{2}) = g^{\prime
\prime}(\xi)(d_{0}-d_{2}) < M \left(\frac{L}{M}\right)= L,
\]
proving \eqref{g-gamma-primeIsNonnegativeInSomeInterval}. Now we consider
two possibilities: either $I_{\gamma} \supseteq[d_{0}-L/M, d_{0}]$ or $I_{\gamma}
\not \supseteq [d_{0}-L/M, d_{0}]$.
Next, we rule out the first one. Suppose by contradiction that $I_{\gamma} \supseteq[d_{0}-L/M, d_{0}]$. Then, we deduce from \eqref{g-gamma-primeIsNonnegativeInSomeInterval} that $g_{\gamma}$ is increasing in
$[d_{0}-L/M, d_{0}]$ and, therefore, $g_{\gamma} \le g_{\gamma}(d_{0})=\gamma<0$ in this interval.
From \eqref{systemForWG}, we get that $w^{\prime}_{\gamma} \le0$,
which implies that $g_{\gamma}^{\prime\prime}\le M$ in $[d_{0}-L/M, d_{0}]$,
according to the definition of $M$. Then, using Taylor's expansion, we obtain that, for some $\xi\in(d_{0} -L/M,d_{0})$,
\begin{align*}
g_{\gamma}\left(  d_{0}-\frac{L}{M} \right)   &  = g_{\gamma}(d_{0}) +
g_{\gamma}^{\prime}(d_{0})\left(  -\frac{L}{M}\right)  + \frac{g_{\gamma
}^{\prime\prime}(\xi)}{2}\left(  -\frac{L}{M} \right)  ^{2}\\[5pt]
&  \le\gamma+ L \left(  -\frac{L}{M}\right)  + \frac{M}{2}\left(  \frac{L}%
{M}\right)  ^{2}\leq -1 , 
\end{align*}
which contradicts $g_{\gamma} > -1$.

Hence the second possibility must occur, that is, $d_{\gamma} \ge d_{0} -
L/M$. From $\ref{cor-b}$ of Corollary \ref{LimitOfg-gamma} and
\eqref{g-gamma-primeIsNonnegativeInSomeInterval}, we deduce that $\lim_{d
\to d_{\gamma}} g_{\gamma}(d) =-1$. Observe also that $d_{\gamma} \ge d_{0} -
L/M > 0$, since $M > L/d_{0}$. Therefore, $\inf_{d > 0, d \in I_{\gamma}}
g_{\gamma}(d)=-1$, that is, $\gamma\not \in A$. This completes the proof.

\end{proof}

\begin{prop}
We have $\gamma_{0} \not \in A$. \label{gamma-zero-notInA}
\end{prop}

\begin{proof} Suppose by contradiction that $\gamma_{0} \in A$. If $d_{\gamma_{0}} <
0$, then $(w_{\gamma_{0}}, g_{\gamma_{0}})$ is defined in $[0,d_{0}]$. For
$\gamma< \gamma_{0}$ sufficiently close to $\gamma_{0}$, due to the continuous
dependence of solutions with respect to the initial conditions, it follows that
\[
| g_{\gamma}(d) - g_{\gamma_{0}}(d) | < \varepsilon\quad\mathrm{for} \quad d
\in[0,d_{0}],
\]
for $$\varepsilon= \displaystyle \frac{ 1 + \inf_{d > 0} g_{\gamma_{0}}(d)}{
2} > 0. $$
Thanks to our choice of $\varepsilon$, we get that $g_{\gamma}(d) > -1 + \varepsilon$ for $d \in[0,d_{0}]$.
Combining this fact with Lemma \ref{gIsBoundedByBelowForDLarge}, we deduce that
\[
\inf_{d > 0} g_{\gamma}(d) \ge\min\left\{  - \frac{1}{2}, \gamma, -1 +
\varepsilon\right\}  > -1.
\]
Hence, $\gamma\in A$ contradicting that $\gamma_{0} = \inf A$. 

Next, let us consider the case $d_{\gamma_{0}} \ge0$. Using Lemma
\ref{characterizationOFA}, we have in this case that $\lim_{d \to d_{\gamma_{0}}}
g_{\gamma_{0}}(d) =1$. Therefore, for some $d_{1} \in(d_{\gamma_{0}}, d_{0})$,
we have
\begin{equation}
g_{\gamma_{0}}(d) > 1/2 \quad \text{ for }  d \in(d_{\gamma_{0}},d_{1}].
\label{gamma-0-d1BiggerThan1-2}%
\end{equation}
Again, using the continuous dependence of solutions on initial conditions, we
get, for $\gamma <\gamma_0$ sufficiently close to $\gamma_0$,
\[
| g_{\gamma}(d) - g_{\gamma_{0}}(d) | <  \min \left\{ 1/4, \varepsilon\right\} , \text{ for }  d \in[d_{1},d_{0}].
\]
 As before, this implies that  $g_{\gamma}(d) > -1+\varepsilon$ for $d \in[d_{1},d_{0}]$ and, together with \eqref{gamma-0-d1BiggerThan1-2}, we conclude that $g_{\gamma}(d_1) > 0$. Hence, from Lemma \ref{boundnessOfTheDerivativesOfg}, we have
that $g_{\gamma}(d) > 0$ for $d \le d_{1}$, $d \in I_{\gamma}$. On the other hand, using 
Lemma \ref{gIsBoundedByBelowForDLarge}, we see that $g_{\gamma} \ge\min\{-1/2, \gamma\}$ in
$[d_{0},+\infty)$. Therefore, we obtain that $\inf_{d \in I_{\gamma}}  g_{\gamma} > -1$.  
Thus, $\gamma\in A$, which contradicts $\gamma_{0}=\inf A$. 
 
\end{proof}

We are now in position to prove that $\gamma_0$ is such that $I_{\gamma_0}=\R^+$.

\begin{thm}
Let $d_{0} >0$ be as in Lemma \ref{gIsBoundedByBelowForDLarge} and
$\gamma_{0}$ defined by \eqref{gammaoDefinition}. Then, $\gamma_{0} < 0$,
$I_{\gamma_{0}} = (0,+\infty)$ (that is, $d_{\gamma_{0}}=0$) and
\[
\lim_{d \to0} g_{\gamma_{0}}(d) = -1.
\]
\label{g-Gamma0IsSol}
\end{thm}

\begin{proof} Recalling that $0\in A$, we deduce, from Propositions \ref{gamma-zero-notInA} and \ref{gamma-zero-biggerThan-1}, that $-1<\gamma_0<0$, and, from Lemma \ref{characterizationOFA} and Corollary \ref{LimitOfg-gamma}, that $d_{\gamma_{0}} \ge0$ and $\lim_{d \to d_{\gamma_{0}}}
g_{\gamma_{0}}(d) = -1.$ Thus, it only remains to prove that $d_{\gamma_{0}}=0$. Suppose by contradiction that $d_{\gamma_{0}} >
0$. Let $\gamma_{k} \in A$ such that $\gamma_{k} < 0$ and $\gamma_{k}
\to\gamma_{0}$. Since the proof is quite lenghty, we split it into three claims.

\ 

\noindent\textsl{Claim 1:} There exists $\delta> 0$ such that, for all $k$ large enough, $(w_{\gamma_{k}%
}, g_{\gamma_{k}})$ is defined in $(d_{\gamma_{0}}-\delta,+\infty)$. Furthermore, $g_{\gamma_{k}} < 0$ in $(d_{\gamma_{0}}-\delta,+\infty)$.

\ 

\noindent\textsl{Proof of Claim 1:} 
Let $D:=n-1+ \max \psi$. We recall from Lemma \ref{boundnessOfTheDerivativesOfg} that $|g_{\gamma_k}^\prime|\leq D$. Next, we set
\[
\delta= \min\{ 1/(4D), d_{\gamma_{0}} \}.
\]
Since $\lim_{d \to d_{\gamma_{0}}} g_{\gamma_{0}}(d) = -1$, there exists $d_{\gamma_0}<d^{*} <
d_{\gamma_{0}}+\delta$ such that $g_{\gamma
_{0}}(d^{*}) < -3/4$. Observe now that from the continuous dependence of
solutions on initial conditions, for $\gamma$ close to $\gamma_0$, it follows that $ [d^{*},d_{0}] \subseteq I_\gamma $ and
\[
| g_{\gamma}(d) - g_{\gamma_{0}}(d) | < 1/4 \text{ for } d
\in[d^{*},d_{0}].
\]
 Hence, using that $\gamma_{k} \to
\gamma_{0}$, we have that $I_{\gamma_{k}}=(d_{\gamma_{k}},+\infty)
\supseteq[d^{*},d_{0}]$ and $-1< g_{\gamma_{k}}(d^{*}) < -1/2$ for $k$ large enough.

Let $d \in(d_{\gamma_{0}} - \delta, d^{*}] \cap I_{\gamma_{k}}$ and notice that $0<
d^{*}-d <(d_{\gamma_{0}} +\delta) - (d_{\gamma_{0}} - \delta)= 2 \delta$. So doing a Taylor's expansion and recalling the definition of $D$, we get that
\[
g_{\gamma_{k}}(d) \le g_{\gamma_{k}}(d^{*}) + D |d^{*} - d| < -\frac{1}{2} + 2
\delta\; D \le-\frac{1}{2} + 2 \frac{1}{4D} D \le0.
\]
Then, for $k$ large enough, we have shown that
\begin{equation}
g_{\gamma_{k}}(d) < 0 \quad\mathrm{for } \quad d \in(d_{\gamma_{0}} - \delta,
d^{*}] \cap I_{\gamma_{k}}. \label{negativityOfG-gamma-k}%
\end{equation}
Suppose now that, for $k$ large enough, $  (d_{\gamma_{0}}-\delta,
+\infty) \not \subseteq I_{\gamma_{k}}$, i.e. $d_{\gamma_{0}}-\delta< d_{\gamma_{k}} < d_{0}$. Using that $\delta\le d_{\gamma_{0}}$, we have
$d_{\gamma_{k}} > 0$. From this and $\gamma_{k} \in A$, Lemma
\ref{characterizationOFA} implies that $\lim_{d \to d_{\gamma_{k}}}
g_{\gamma_{k}}(d) = 1.$ However, this contradicts
\eqref{negativityOfG-gamma-k}. Then $(d_{\gamma_{0}%
}-\delta, +\infty)  \subseteq I_{\gamma_{k}}$ and $g_{\gamma_{k}} < 0$ in $(d_{\gamma_{0}} - \delta,
d^{*}]$. Hence, the last statement of Lemma \ref{boundnessOfTheDerivativesOfg}
implies that $g_{\gamma_{k}} < 0$ in $(d_{\gamma_{0}} - \delta, +\infty)$
proving the claim. 

\ 

From now on, we only consider $k$ for which this claim holds. Observe also
that $|g_{\gamma_{k}}| < 1$ and from Lemma
\ref{boundnessOfTheDerivativesOfg},  we have that $|g_{\gamma_{k}}^{\prime}| <
D$ in $I_{\gamma_{k}} \supset(d_{\gamma_{0}}-\delta, +\infty)$. Then, applying Arzel\`a-Ascoli Theorem, there exists a subsequence,
which we denote by $g_{\gamma_{k}}$, that converges uniformly,  in $[d_{\gamma_{0}},d_{0}]$, to some
continuous function $g$. On the other hand, from the continuous
dependence of solutions on initial conditions and $\gamma_{k} \to\gamma_{0}$,
we get that $g_{\gamma_{k}} \to g_{\gamma_{0}}$ uniformly in compact subsets
of $(d_{\gamma_{0}}, d_{0}]$. Therefore, we deduce that $g=g_{\gamma_{0}}$ in $(d_{\gamma_{0}%
},d_{0}]$ and $g(d_{\gamma_{0}})= \lim_{d \to d_{\gamma_{0}}}g_{\gamma_{0}}(d)
= -1$ which implies that
\begin{equation}
\lim_{k \to+\infty} g_{\gamma_{k}}(d_{\gamma_{0}}) = -1.
\label{limitOfGkAtD-gamma-0}%
\end{equation}
Moreover, since $g_{\gamma_{k}}^{\prime}(d_{\gamma_{0}})$ is bounded from
Lemma \ref{boundnessOfTheDerivativesOfg}, there exists $\alpha \in \R$ such that, up to a subsequence, we have
\begin{equation}
\lim_{k \to+\infty} g_{\gamma_{k}}^{\prime}(d_{\gamma_{0}}) = \alpha .
\label{limitOfgkPrimeAtD-gamma-0}%
\end{equation}

\ 

\noindent\textsl{Claim 2:} We have $\alpha= 0$.

\ 

\noindent  \textsl{Proof of Claim 2:}  Suppose that $\alpha> 0$, the case $\alpha<0$ follows in the same way. We follow the same idea as in Proposition
\ref{gamma-zero-biggerThan-1}. First, observe that $g_{\gamma_{k}} < 0$ in
$(d_{\gamma_{0}} -\delta, +\infty)$ from Claim 1. Hence $w_{\gamma_{k}%
}^{\prime}< 0$ in this interval from \eqref{systemForWG} and, therefore, Lemma
\ref{boundnessOfTheDerivativesOfg} implies that there exists $M > 0$ such
that
\[
g_{\gamma_{k}}^{\prime\prime}\le M \text{ in }(d_{\gamma_{0}}
-\delta, +\infty).
\]
Now let $d \in(d_{\gamma_{0}}-\alpha /2M, d_{\gamma_{0}})$ and define $\rho=d_{\gamma_{0}} - d >0$. From
\eqref{limitOfGkAtD-gamma-0} and \eqref{limitOfgkPrimeAtD-gamma-0}, there
exists $k$ such that
\[
g_{\gamma_{k}}(d_{\gamma_{0}}) < -1 + \frac{\alpha\rho}{4} \text{ and }
 g_{\gamma_{k}}^{\prime}(d_{\gamma_{0}}) > \frac{\alpha}{2} .
\]
Hence, using a Taylor's expansion, we have, for some
$\xi\in(d,d_{\gamma_{0}})$,
\begin{align*}
g_{\gamma_{k}}\left(  d \right)   &  = g_{\gamma_{k}}(d_{\gamma_{0}}) +
g_{\gamma_{k}}^{\prime}(d_{\gamma_{0}})\left(  -\rho\right)  + \frac
{g_{\gamma_{k}}^{\prime\prime}(\xi)}{2}\left(  -\rho\right)  ^{2}\\[5pt]
&  \le-1 + \frac{\alpha\rho}{4} - \frac{\alpha}{2} \rho+ \frac{M}{2} \rho
^{2}<-1
\end{align*}
contradicting $g_{\gamma_{k}} > -1$. This proves the claim.

\ 

\noindent\textsl{Claim 3:} There holds $\displaystyle \lim_{k \to\infty} w_{\gamma_{k}%
}(d_{\gamma_{0}}) = +\infty$.

\ 

\noindent  \textsl{Proof of Claim 3:} Let $M$ as in Claim 2 and $d^{*} \in(d_{\gamma_{0}}, d_{0})$ such
that $(d^{*}-d_{\gamma_{0}})^{2} < \min\{1,1/2M\}$. Given $0 < \varepsilon<
1/4$, we deduce from \eqref{limitOfGkAtD-gamma-0}, \eqref{limitOfgkPrimeAtD-gamma-0} and
$\alpha=0$, that, for $k$ large enough,
\[
g_{\gamma_{k}}(d_{\gamma_{0}}) < -1 + \frac{\varepsilon}{2} \text{ and }
 g_{\gamma_{k}}^{\prime}(d_{\gamma_{0}}) < \frac{\varepsilon}{2}.
\]
Using a Taylor's expansion again for $d \in[d_{\gamma_{0}}, d^{*}]$, there exists
$\xi\in(d_{\gamma_{0}},d)$ such that
\begin{align*}
g_{\gamma_{k}}\left(  d \right)   &  = g_{\gamma_{k}}(d_{\gamma_{0}}) +
g_{\gamma_{k}}^{\prime}(d_{\gamma_{0}})\left(  d-d_{\gamma_{0}} \right)  +
\frac{g_{\gamma_{k}}^{\prime\prime}(\xi)}{2}\left(  d-d_{\gamma_{0}} \right)
^{2}\\[5pt]
&  < -1 + \varepsilon + \frac{M}{2}
\left(  d-d_{\gamma_{0}} \right)  ^{2}\leq - 1/2, 
\end{align*}
and, therefore,
$$1-(g_{\gamma_{k}}(d))^{2}
\le 4[\varepsilon+ M(d-d_{\gamma_{0}})^{2}].$$
 Then, using that $g_{\gamma_{k}} < 0$ in
$(d_{\gamma_{0}}-\delta, +\infty)$ and $d^{*} < d_{0}$, we get
\begin{align*}
\int_{d_{0}}^{d_{\gamma_{0}}} \!\!\frac{g_{\gamma_{k}}(t)}{\sqrt
{1-(g_{\gamma_{k}}(t))^{2}}}\, dt  &  = \int_{d_{\gamma_{0}}}^{d_{0}}
\!\!\frac{- g_{\gamma_{k}}(t)}{\sqrt{1-(g_{\gamma_{k}}(t))^{2}}} \, dt\\[5pt]
&  \ge\int_{d_{\gamma_{0}}}^{d^{*}} \! \!\frac{- g_{\gamma_{k}}(t)}%
{\sqrt{1-(g_{\gamma_{k}}(t))^{2}}} \, dt\\[5pt]
&  \ge\frac{1}{4} \int_{d_{\gamma_{0}}}^{d^{*}}\!\! \frac{1}{\sqrt
{\varepsilon+ M(t-d_{\gamma_{0}})^{2}}} \, dt.
\end{align*}
Since the last expression goes to infinity as $\varepsilon\to0$, we conclude
that
\[
\lim_{k \to\infty} \int_{d_{0}}^{d_{\gamma_{0}}} \!\!\frac{g_{\gamma_{k}}%
(t)}{\sqrt{1-(g_{\gamma_{k}}(t))^{2}}}\, dt =+\infty.
\]
Then, from \eqref{wExpressionByG}, it follows that $w_{\gamma_{k}}%
(d_{\gamma_{0}}) \to+\infty$ as $k \to\infty$, proving the claim. 

\ 

We are now in position to complete the proof of the theorem.\\

\noindent  \textsl{End of the proof of Theorem \ref{g-Gamma0IsSol}:} Remind that, from
\eqref{systemForWG}, we have
\[
g_{\gamma_{k}}^{\prime}(d_{\gamma_{0}}) = -(n-1)\tanh(d_{\gamma_{0}}) \;
g_{\gamma_{k}}(d_{\gamma_{0}}) - \psi(d_{\gamma_{0}}, w_{\gamma_{k}}(d_{\gamma_{0}})).
\]
Now, observe that Claim $3$ and \ref{psiDefinitionef} implies that $\lim_{k \to\infty} \psi(d_{\gamma_{0}}, w_{\gamma_{k}}(d_{\gamma_{0}})) =0$. Hence, letting $k \to\infty$ in this equation and using \eqref{limitOfGkAtD-gamma-0},
\eqref{limitOfgkPrimeAtD-gamma-0} and Claim 2, we conclude that
\[
(n-1)\tanh(d_{\gamma_{0}}) =0.
\]
However, this contradicts the assumption $d_{\gamma_{0}} > 0$. Therefore,
$d_{\gamma_{0}}=0$, completing the proof. 
\end{proof}

Next, we prove the uniqueness of such $\gamma_0$ in the following sense:

\begin{prop}
If $\gamma\in(-1,1)$ is such that $d_{\gamma}=0$ and $\lim_{d \to0}g_{\gamma
}(d)=-1$, then $\gamma=\gamma_{0}$. In other words, $\gamma_{0}$ is the unique
$\gamma$ such that $I_{\gamma}=(0,+\infty)$ and $g_{\gamma}(0^{+})=-1$.
\label{uniquenessOfGamma-0}
\end{prop}

\begin{proof}
The proof follows by contradiction. Assume that there
exists $\gamma_{1} \ne\gamma_{0}$ such that $d_{\gamma_{1}} =0$. Suppose without loss of generality
that $\gamma_{1} > \gamma_{0}$. By definition, we have that $g_{\gamma_{1}}(d_{0}) =
\gamma_{1} > \gamma_{0} = g_{\gamma_{0}}(d_{0})$. From the continuity of
$g_{\gamma_{i}}$, it follows that $g_{\gamma_{1}} > g_{\gamma_{0}}$ in some
neighborhood of $d_{0}$. Let $J$ be the largest open interval contained in
$(0,+\infty)$ such that $d_{0} \in J$ and $g_{\gamma_{1}} > g_{\gamma_{0}}$ in
$J$. Using that the application $z \mapsto z/\sqrt{1-z^{2}}$ is
increasing on $(-1,1)$, we obtain
\[
\frac{g_{\gamma_{1}}(t)}{\sqrt{1-g_{\gamma_{1}}^{2}(t)}} > \frac{g_{\gamma
_{0}}(t)}{\sqrt{1-g_{\gamma_{0}}^{2}(t)}} \text{ for }  t \in J.
\]
Hence, from \eqref{wExpressionByG}, we conclude that $w_{\gamma_{1}}(d) <
w_{\gamma_{0}}(d)$ if $d < d_{0}$ and $d \in J$. Then, using
\eqref{systemForWG} and \ref{psiDefinitionc}, we have
\begin{align}
g_{\gamma_{1}}^{\prime}(d)  &  = -(n-1)\tanh(d) \; g_{\gamma_{1}}(d) - \psi(d,
w_{\gamma_{1}}(d))\nonumber\\[5pt]
&  < -(n-1)\tanh(d) \; g_{\gamma_{0}}(d) - \psi(d, w_{\gamma_{0}}(d)) =
g_{\gamma_{0}}^{\prime}(d) , \label{g-gamma-1SmallerThanG-gamma-0}%
\end{align}
for $d < d_{0}$ ($d \in J$).

Let $d^{*}$ be the left endpoint of the interval of $J$. Thus $0 \le d^{*} <
d_{0}$. If $d^{*} > 0$, then $g_{\gamma_{0}}$ and $g_{\gamma_{1}}$ are defined
at $d^{*}$ and, therefore, $g_{\gamma_{1}}(d^{*})=g_{\gamma_{0}}(d^{*})$ due
to the definition of $J$. Since $g_{\gamma_{1}}(d^{*})=g_{\gamma_{0}}(d^{*})$
and $g_{\gamma_{1}} > g_{\gamma_{0}}$ in $J$, there exists some $\xi\in
(d^{*},d_{0})$ such that $g_{\gamma_{1}}^{\prime}(\xi) > g_{\gamma_{0}%
}^{\prime}(\xi)$. But, this contradicts \eqref{g-gamma-1SmallerThanG-gamma-0}.
Then $d^{*}=0$. Hence, we have $g_{\gamma_{1}}(0^{+})=-1=g_{\gamma_{0}}%
(0^{+})$ and $g_{\gamma_{1}} > g_{\gamma_{0}}$ in $(0,d_{0}]$. As before, we
get a contradiction with \eqref{g-gamma-1SmallerThanG-gamma-0}. 
\end{proof}

We are also able to prove that $w_{\gamma_0}(d)$ blows up when $d\rightarrow 0^+$ and therefore it satisfies the boundary condition on $S$ of \eqref{superScherkEDP}.

\begin{lem}
Let $\gamma_0$ be defined as in \eqref{gammaoDefinition}. Then, we have
\[
\lim_{d\rightarrow0}w_{\gamma_{0}}(d)=+\infty.
\]
\label{w-omega-goes-to-infinity-at-0}
\end{lem}

\begin{proof} This lemma will follow by comparing our solution to an ordinary Scherk graph namely  a solution of
\begin{equation}
\begin{cases}
W^{\prime}(d) = \displaystyle \frac{G(d)}{\sqrt{1-G^{2}(d)}}\\[10pt]
G^{\prime}(d) = -(n-1)\tanh(d) \; G(d),
\end{cases}  \label{systemForSomeScherkInB}%
\end{equation}
with initial conditions
\[
W(d_{0})=h \text{ and } G(d_{0})= \frac{-1}{\cosh^{n-1}(d_{0})}.
\]
Observe that the previous system can be solved explicitly by
\begin{equation}
G(d) = \frac{-1}{\cosh^{n-1}(d)}\text{ and } W(d) = h -
\int_{d_{0}}^{d} \frac{1}{\sqrt{\cosh^{2n-2}(t) - 1}}\, dt.
\label{ExpressionForScherkInB}%
\end{equation}
 Also notice that
\[
(G (d)\cosh^{n-1}(d))^{\prime}= 0 \quad\mathrm{in} \quad(0,+\infty).
\]
On the other hand, from
\eqref{detivativeOfProductGammaCosh}, we have
\[
(g_{\gamma_{0}} (d)\cosh^{n-1}(d))^{\prime}\le0 \quad\mathrm{in} \quad
I_{\gamma_{0}} =(0,+\infty).
\]
Therefore, defining $H(d)= g_{\gamma_{0}} (d)\cosh^{n-1}(d) - G
(d)\cosh^{n-1}(d)$, it follows that $H$ is nonincreasing in $(0,+\infty),$ and
\[
\lim_{d \to0} H(d)= \lim_{d \to0} g_{\gamma_{0}} (d)\cosh^{n-1}(d) - G
(d)\cosh^{n-1}(d) =0.
\]
Hence, we deduce that $H(d) \le0$ for any $d \in(0,+\infty)$ which implies that
\begin{equation}
g_{\gamma_{0}} \le G \quad\mathrm{in} \quad(0,+\infty).
\label{g0BiggerThanG-gamma-0}%
\end{equation}
Then, since the map $z \mapsto z/\sqrt{1-z^{2}}$ is increasing in $(-1,1)$, we have
\[
w_{\gamma_{0}}(d) = h + \int_{d_{0}}^{d} \frac{g_{\gamma_{0}}(t)}%
{\sqrt{1-g_{\gamma_{0}}^{2}(t)}}\; dt \ge h + \int_{d_{0}}^{d} \frac{G%
(t)}{\sqrt{1-G^{2}(t)}}\; dt = W(d)
\]
for $d \le d_{0}$. Since $\lim_{d\to0} W (d) = +\infty$, we
conclude that $\lim_{d\to0} w_{\gamma_{0}}(d) = +\infty$. 
 
\end{proof}

Until now, we have proved that if $d_{0}>0$ satisfies \eqref{d0Condition1} and
$h \in\mathbb{R}$, then there exist $\gamma_{0}=\gamma_{0}(d_{0}, h)$ and a
solution $w_{\gamma_{0}}=w_{d_{0},h,\gamma_{0}}$ to the equation
\eqref{superScherkEDO} in $(0,+\infty)$ with initial condition $w_{\gamma_{0}}(d_0)=h.$
Besides $w_{\gamma_{0}}(0)=+\infty$.
Hence, it still remains to show that $w_{\gamma_{0}}(+\infty)=c$ for some such $d_{0}$ and $h \in\mathbb{R}$. Choose and fix such $d_{0}$. We analyze the relation between $h$ and
$w_{d_{0},h,\gamma_{0}}$ for this $d_{0}$. Since $d_{0}$ is fixed and
$\gamma_{0}$ depends on $h$, we use the notation $\gamma_{0}(h)=\gamma
_{0}(d_{0},h).$ Our goal now is to show that the application $\ell(h)=\lim_{d\to \infty} w_{d_{0},h,\gamma_{0}(h)}(d)$ is well-defined, continuous in $\mathbb{R}$ and surjective on $\mathbb{R}$.

\ 

\begin{lem}
The map $h\mapsto\gamma_{0}(h)$ is continuous. \label{continuityOfH}
\end{lem}

\begin{proof}
Suppose by contradiction that $\gamma_{0}(h)$ is not continuous at
some $h^{*} \in\mathbb{R}$. Then, there exists a sequence $(h_{k})$ and  $\varepsilon >0$ such that
\begin{equation}
h_{k} \to h^{*} \quad\mathrm{and} \quad|\gamma_{0}(h_{k}) - \gamma_{0}(h^{*})|
> \varepsilon. \label{differenceBetweenGammaHkAndGammaHstar}%
\end{equation}
Observe that, from Proposition \ref{gamma-zero-biggerThan-1} and Theorem
\ref{g-Gamma0IsSol}, $-1+\delta\le \gamma_{0}(h_{k}) < 0$. Hence, up to a subsequence, there exists $\gamma^{*}\in[-1+\delta,0]$ such that $\lim_{k\rightarrow \infty}\gamma_{0}(h_{k})=\gamma^{*}$. 
Consider the solution ($w_{k},g_{k}$) of \eqref{systemForWG} associated to
$(\gamma_{0}(h_{k}))$, that is, $w_{k}(d_{0})=h_{k}$ and $g_{k}(d_{0}%
)=\gamma_{0}(h_{k})$. Observe that the domain of ($w_{k},g_{k}$) is $(0,+\infty)$
and $g_{k}(0^{+})=-1$, thanks to Theorem \ref{g-Gamma0IsSol}. Moreover, $|g_{k}| < 1$ and $|g_{k}^{\prime}|$ is
bounded by Lemma \ref{boundnessOfTheDerivativesOfg}. Hence, extending
$g_{k}$ continuously at $0$ by $g_{k}(0)=-1$, we can apply Arzel\`a-Ascoli
Theorem to conclude that some subsequence of $(g_{k})$ converges uniformly to
some continuous function $\tilde{g}$ in $[0,d_{0}]$. Observe that
\[
\tilde{g}(0)=\lim_{k \to\infty} g_{k}(0)=-1 \text{ and }\tilde
{g}(d_{0})=\lim_{k \to\infty} g_{k}(d_{0})=\gamma^{*}.
\]
On the other hand, from the continuous dependence of solutions on initial
conditions, $h_{k} \to h^{*}$ and $\gamma_{0}(h_{k}) \to\gamma^{*}$, it
follows that $w_{k} \to w^{*}$ and $g_{k} \to g^{*}$ uniformly in compacts of
$I_{h^{*},\gamma^{*}}$, where ($w^{*},g^{*}$) is the solution of
\eqref{systemForWG} in $I_{h^{*},\gamma^{*}}$ such that $w^{*}(d_{0})=h^{*}$
and $g^{*}(d_{0})=\gamma^{*}$. Therefore, $\tilde{g}=g^{*}$ in $[0,d_{0}] \cap
I_{h^{*},\gamma^{*}}$. If $0 \in I_{h^{*},\gamma^{*}}$, then $g^{*}%
(0)=\tilde{g}(0)=-1$, contradicting that $|g^{*}| < 1$ in $I_{h^{*},\gamma
^{*}}$. Hence, the left endpoint of $I_{h^{*},\gamma^{*}}$, denoted by $d^{*}%
$, satisfies $0\le d^{*} < d_{0}$. 

 Next, notice that
\[
\lim_{d\to d^{*}}g^{*}(d)=-1,
\]
otherwise, according to Corollary \ref{LimitOfg-gamma}, $\lim_{d\to d^{*}%
}g^{*}(d)=1$, which contradicts $g_{k} \to g^{*}$ and $g_{k} < 0$ from $\ref{cor-d}$
of Corollary \ref{LimitOfg-gamma}. Hence, the solution $g_{k}:(0,+\infty)
\to\mathbb{R}$ converges uniformly to the solution $g^{*}:(d^{*},+\infty)
\to\mathbb{R}$ in $(d^{*},d_{0}]$, and $\lim_{k\rightarrow \infty} g_{k}(d^{*})=\lim_{d\to
d^{*}}g^{*}(d)=-1$. Following the same argument as in Claims 2 and 3 of Theorem
\ref{g-Gamma0IsSol}, one can show that $d^{*}=0$
and $g^{*}$ is a solution of \eqref{systemForWG} in $(0,+\infty)$ satisfying
$g^{*}(d_{0})=\gamma^{*}$ and $w^{*}(d_{0})=h^{*}$. Moreover, $\lim_{d\to0}
g^{*}(d)=-1$, since $g_{k} \to g^{*}$ and $g_{k} < 0$. Hence, from Proposition
\ref{uniquenessOfGamma-0}, $\gamma_{0}(h^{*})=\gamma^{*}$. But, this
contradicts $\gamma_{0}(h_{k}) \to\gamma^{*}$ and
\eqref{differenceBetweenGammaHkAndGammaHstar}, proving that $\gamma_{0}$ is
continuous. 
\end{proof}

\begin{lem}
For any $h\in\mathbb{R}$, there exists $\displaystyle{\lim_{d\rightarrow+\infty}w_{d_{0},h,\gamma_{0}}(d)}.$ We denote this limit by $\ell(h).$
Moreover, the convergence is uniform in $h$.
\label{estimateBetweenW-h-dAndTheLimit}
\end{lem}

\begin{proof} To simplify notation, we denote $g_{\gamma_{0}}(d)=g_{\gamma_{0}(h)}(d)$ by
$g_{h}(d)$. First recall that $g_{h} <0$ from $\ref{cor-d}$ of Corollary
\ref{LimitOfg-gamma}. 
Moreover, from Lemma \ref{gIsBoundedByBelowForDLarge} and Proposition
\ref{gamma-zero-biggerThan-1}, we deduce that
there exists $\beta\in(0,1)$ that does not depend on $h$ such that $g_{h}(d) \ge-\beta$ for $d
\ge d_{0}$. 
Therefore, we have
\begin{equation}
\frac{1}{\sqrt{1-g_{h}^{2}(d)}} \le C_{1}:=\frac{1}{\sqrt{1 -\beta^{2}}}
\text{ for } d \ge d_{0}. \label{g-hSupBoundForInverse}%
\end{equation}
We also recall that, according to \eqref{eq-defrho}, there holds
\begin{equation}
|g_{h}(d)| \le\frac{\cosh^{n-1}(d_{0})}{\cosh^{n-1}(d)} + \rho(d)
\text{ for }  d \ge d_{0}. \label{g-hSupBound}%
\end{equation}
Since $d_{0}$ is fixed and
$\gamma_{0}$ depends on $h$, we use the notation  $w_{h}(d) = w_{d_{0},h,\gamma_{0}}(d)$.
Using \eqref{wExpressionByG}, the
inequalities \eqref{g-hSupBound} and \eqref{g-hSupBoundForInverse} imply
\[
|w_{h}(d_{2})-w_{h}(d_{1})| = \left|  \int_{d_{1}}^{d_{2}} \frac{g_{h}%
(t)}{\sqrt{1-g_{h}^{2}(t)}} \; dt \right|  \le C_{1} \int_{d_{1}}^{d_{2}}
\frac{\cosh^{n-1}(d_{0})}{\cosh^{n-1}(t)}  +  \rho(t) \; dt,
\]
for $d_{0} \le d_{1} < d_{2}$.
So using \eqref{boundForIntegralOfRho}, we obtain that
\begin{equation}
|w_{h}(d_{2})-w_{h}(d_{1})| \le C_{0} \left(  \frac{1}{\cosh^{n-1}(d_{1}) } +
\rho(d_{1}) + \int_{d_{1}}^{d_{2}} \psi(s,0) \, ds \right)  ,
\label{differenceEstimateBetweenWh1-Wh2}%
\end{equation}
for some constant $C_0$  depending only on $n$, $d_{0}$ and $\psi$.
From this, we get that $w_{h}$ is bounded in $[d_{1},+\infty)$.
Furthermore, since $g_{h} < 0$ and $w_{h}$ satisfies \eqref{wExpressionByG}, we deduce that $w_{h}$ is decreasing. 
Thus, $w_h(d)$ converges as $d \to \infty$. The uniform convergence is due to the fact that the right-hand side of \eqref{differenceEstimateBetweenWh1-Wh2} does not depend on $h$.
This establishes the proof.
\end{proof}

We are now in position to prove the following theorem:

\begin{thm}
\label{lemsurjective}
The application $h\mapsto\ell(h)$ is continuous in $\mathbb{R}$ and
surjective on $\mathbb{R}$. \label{continuityOfOmega-h}
\end{thm}

\begin{proof} Let $h_{0} \in\mathbb{R}$ and $\varepsilon> 0$. By Lemma
\ref{estimateBetweenW-h-dAndTheLimit}, we have, for $d_{1} > d_{0}$ large enough,
\[
|w_{h}(d_{1}) -\ell(h)| < \varepsilon/3 \text{ for any } h
\in\mathbb{R}.
\]
On the other hand, from Lemma \ref{continuityOfH}, if $h$ is close to $h_{0}$, then
$\gamma_{0}(h)$ is close to $\gamma_{0}(h_{0})$. Hence, using the continuous
dependence of solutions on initial conditions, there exists $\delta_{1} >0$
such that if $|h-h_{0}| < \delta_{1}$, then
\[
|w_{h}(d) - w_{h_{0}}(d)| = |w_{d_{0},h,\gamma_{0}(h)}(d) - w_{d_{0}%
,h_{0},\gamma_{0}(h_{0})}(d)| < \varepsilon/3 \text{ for }d
\in[d_{0},d_{1}].
\]
In particular,
\[
|w_{h}(d_{1}) - w_{h_{0}}(d_{1})| < \varepsilon/3.
\]
Therefore, if $|h-h_{0}| < \delta_{1}$, we have
\[
|\ell(h) - \ell(h_{0})| \le|\ell(h) - w_{h}(d_{1})| + |w_{h}(d_{1}) -
w_{h_{0}}(d_{1})| + |w_{h_{0}}(d_{1}) -\ell(h_{0})| < \varepsilon.
\]
Thus, the application $h\mapsto\ell (h)$ is continuous. To prove
that $\ell$ is surjective, let
\[
\sigma:= C_{0} \left(  \frac{1}{\cosh^{n-1}(d_{0}) } + \rho(d_{0}) +
\int_{d_{0}}^{+\infty} \psi(s,0) \, ds \right)  .
\]
Then, from Lemma \ref{estimateBetweenW-h-dAndTheLimit} (see \eqref{differenceEstimateBetweenWh1-Wh2}), we have
\[
|w_{h}(d_{0}) - \ell(h) | \le\sigma \text{ for any }  h
\in\mathbb{R}.
\]
Recalling that $w_{h}(d_{0})=h$, we deduce that $h - \sigma\le\ell(h) \le h + \sigma$
and, therefore,
\[
\lim_{h \to+\infty} \ell(h) = +\infty\text{ and }\lim_{h
\to-\infty} \ell(h) = -\infty.
\]
The continuity of $\ell$ implies that $\ell$ is an onto
application. 
\end{proof}

\begin{cor}
For any $c\in\mathbb{R}$, problem \eqref{superScherkEDO} and, therefore,
problem \eqref{superScherkEDP} have a solution. The solution $w$ of problem \eqref{superScherkEDO} is a decreasing function in $(0,\infty),$ satisfying $w>c.$
Even if $f$ changes sign, this solution is a supersolution of \eqref{ScherkProblem}. \label{MainTheorem}
\end{cor}

\begin{proof}
According to Theorem \ref{continuityOfOmega-h},
$\ell(h)$ is surjective on $\mathbb{R}$. Then, there exists $h_{c}$ such
that $\ell(h_{c})=c$. Therefore, the solution $w_{h_{c}}(d)=w_{d_{0},h_{c},\gamma_{0}(h_{c})}(d)$ associated to $h_{c}$ is the one that we are
looking for. 

Remind that $g_{\gamma_0}=g_{d_0,h_c,\gamma_0(h_c)}$ that corresponds to $w_{h_{c}}=w_{d_{0}
,h_{c},\gamma_{0}(h_{c})}$ satisfies
$$ \lim_{d \to0} g_{\gamma_{0}}(d) = -1 ,$$
according to Theorem \ref{g-Gamma0IsSol}. Hence, from the last item of Lemma \ref{boundnessOfTheDerivativesOfg}, $g_{\gamma_{0}} < 0$ in $[0,+\infty)$.
Therefore, $w_{h_{c}}'(d)<0$ for any $d >0$, that is, the solution of \eqref{superScherkEDO}, $w_{h_{c}}$ is a decreasing function. Then, using that  $w_{h_c}(+\infty)=c$, we conclude that $ w_{h_c} > c$. 

\end{proof}

By a standard comparison argument, we also have:

\begin{proposition}
There exists a subsolution $W:[0,+\infty) \to \R$ for \eqref{superScherkEDO}, given by \eqref{ExpressionForScherkInB}, satisfying $W(+\infty)=c$ and $W(d)\leq w_{h_{c}}(d)$, where $w_{h_{c}}$ is the solution from previous corollary.
\label{kindOfScherkBelowSol}
\end{proposition}

\begin{proof}
[Proof of Theorem \ref{s}]Let $W$ and $w_{h_c}$ be as stated in Proposition \ref{kindOfScherkBelowSol}.
Then they are sub and supersolutions of $Q(v)=f(x,v)$ in $B$ such that $W,w_{h_c}=+\infty$ on $S$ and
$W,w_{h_c}=c$ on $\partial_{\infty}B.$ Set
\[
\mathcal{S}_{B}=\left\{  \sigma\in C^{0}\left(  B\right) \, |\,
\sigma\text{ is a subsolution of }Q\text{ and }W\leq\sigma\leq w_{h_c}\right\}  .
\]
From the results from Section \ref{sec-genExist}, it is clear that Perron's method can be
applied to the equation $Q(v)=f(x,v)$ and we conclude that the function $u$ defined in
$B$ by
\[
u(x)=\sup_{\sigma\in\mathcal{S}_{B}}\sigma(x),\text{ }x\in B,
\]
is $C^{2}$ and satisfies $Q(u)=f(x,u).$ Clearly $u|_{S}=+\infty,$ $u$ extends
continuously to $\partial_{\infty}B$ and $u|_{\partial_{\infty}B}=c$, since $W \le u \le w_{h_c}$.

Observe that $w_{h_c}$ is a supersolution of \eqref{ScherkProblem} even if $f$ changes sign and satisfies only conditions \ref{propphi1} and \ref{propphi2}. If $\tilde{w}$ is a supersolution of $Q(v)=-f(x,v)$ that satisfies $v=-c$ on $\partial_{\infty} B$ and $v=+\infty$ on $S$, then $-\tilde{w}$ is a subsolution of \eqref{ScherkProblem} that satisfies $v=-\infty$ on $S$. 
\end{proof}

\section{The asymptotic Dirichlet problem}
\label{The_asymptotic_Dirichlet_problem_subsection}

In this section, we solve the asymptotic Dirichlet problem for our prescribed mean curvature type equation by making use of the Scherk type solutions obtained in the previous section. We refer to \cite{CHHfmin} for related results using another method.

\begin{proof}[Proof of Theorem \ref{thm-Scherk}]
First consider the case where $f$ satisfies condition \ref{propphi2}. We use Perron's method by setting
\[
\mathcal{S}_{\varphi}=\left\{  \sigma\in C^{0}\left(  \mathbb{H}^{n}\right)\,
\vert\, \sigma\text{ is a subsol. and }\limsup_{x\rightarrow x_{0}}\sigma
(x)\leq\varphi(x_{0}),\,x_{0}\in\partial_{\infty}\mathbb{H}^{n}\right\}
\]
and defining
\[
u(x)=\sup_{\sigma\in\mathcal{S}_{\varphi}}\sigma(x),\text{ }x\in\mathbb{H}%
^{n}.
\]
We first prove that $u$ is well defined. Set
\[
C_{1}=\min_{\partial_{\infty}\mathbb{H}^{n}} \varphi 
\text{ and }C_{2}=\max_{\partial_{\infty}\mathbb{H}^{n}} \varphi
 .
\]
Let $S_{1}\subset B_2$ and $S_{2}\subset B_1$ be two totally geodesic hypersurfaces of
$\mathbb{H}^{n}$ where each $B_i$, $i=1,2,$ is a connected component of
$\mathbb{H}^{n}\backslash S_{i}$. By Theorem \ref{s}, let $w_{i}$ be two Scherk type supersolutions of \eqref{ScherkProblem} in $B_{i}$ such that $w_{i}|_{\partial_{\infty}B_{i}}=C_{2}$ and $w_{i}|_{S_{i}}=+\infty.$ Let $v_{i}$ be two Scherk type subsolutions of \eqref{ScherkProblem} in $B_{i}$ such that $v_{i}
|_{\partial_{\infty}B_{i}}=C_{1}$ and $v_{i}|_{S_{i}}=-\infty.$ Define super and
subsolutions $w,v$ in $\mathbb{H}^{n}$ by
\[w(x)=\begin{cases}
w_{1}(x)\text{ if }x\in B_{1}\backslash B_{2}\\
\inf\left\{  w_{1}(x),w_{2}(x)\right\}  \text{ if }x\in B_{1}\cap B_{2}\\
w_{2}(x)\text{ if }x\in B_{2}\backslash B_{1}
\end{cases}\]
and
\[v(x)=\begin{cases}
v(x)\text{ if }x\in B_{1}\backslash B_{2}\\
\sup\left\{  v_{1}(x),v_{2}(x)\right\}  \text{ if }x\in B_{1}\cap B_{2}\\
v_{2}(x)\text{ if }x\in B_{2}\backslash B_{1}.\end{cases}\]
Since $v\in\mathcal{S}_{\varphi}$ by the comparison principle we have that
$\sigma\leq w$ for all $\sigma\in\mathcal{S}_{\varphi}.$ It follows that $u$
is well defined. Perron's method then guarantees that $u\in C^{2}\left(
\mathbb{H}^{n}\right)  $ and that $Q(u)(x)=f(x,u)$ for all $x\in\mathbb{H}
^{n}.$ We now prove that $u$ extends continuously to $\partial_{\infty}
\mathbb{H}^{n}$ and that $u|_{\partial_{\infty}\mathbb{H}^{n}}=\varphi.$

 Choose $x_{0}\in\partial_{\infty}\mathbb{H}^{n}$ and let $\varepsilon>0$ be given. Since $\varphi$ is continuous, there exists an open neighborhood
$W\subset\partial_{\infty}\mathbb{H}^{n}$ of $x_{0}$ such that $\varphi
(x)<\varphi(x_{0})+\varepsilon$ for all $x\in W$. We may take a totally
geodesic hypersuface $S$ of $\mathbb{H}^{n}$ such that one of the connected
components of $\mathbb{H}^{n}\backslash S,$ say $B,$ is such that $x_{0}%
\in\partial_{\infty}B\subset W.$ Define a Scherk type solution $\tilde{w}$
on $B$ such that $\tilde{w}|_S=+\infty$ and $\tilde{w}|_{\partial
_{\infty}B}=C,$ where $C=\varphi(x_{0})+\varepsilon.$ Note that given
$\sigma\in\mathcal{S}_{\varphi},$  and denoting by $\sigma_{B}=\sigma|_B,$ the
comparison principle implies that $\tilde{w}|_B\leq\sigma_{B}.$ It follows
that $u\leq\tilde{w}$ in $B$ and then
\[
\limsup_{x\rightarrow x_{0}}u(x)\leq\varphi(x_{0})+\varepsilon.
\]
We may also construct a subsolution $\tilde{v}\in\mathcal{S}_{\varphi}$
such that $\tilde{v}|_{\partial_{\infty}B}=C,$ where now $C=\varphi
(x_{0})-\varepsilon$ so that
\[
\liminf_{x\rightarrow x_{0}}u(x)\geq\varphi(x_{0})-\varepsilon.
\]
Since $\varepsilon>0$ is arbitrary we obtain that $\lim_{x\rightarrow x_{0}%
}u(x)=\varphi(x_{0}),$ concluding the proof when condition \ref{propphi2} is satisfied.

If, instead of condition \ref{propphi2}, we assume $\phi(r)\leq (n-1)\coth(r),$ we obtain the following a priori bound for a solution $u$ to the \eqref{eq-asymDP}. Let $o\in \Hn$ be the point in condition \ref{propphi1}. Let $v:[0,\infty) \to \R$ be the solution of the following ODE
\begin{equation}
\begin{cases}
Q(v\circ r)= \phi(r)\\
v'(0)=0\\
v(\infty) = M,
\end{cases}  
\end{equation}

where $M=\displaystyle{\sup_{\partial_\infty \Hn}\varphi}.$ Then
$$v(r)=\int_r^{+\infty} \frac{\tilde{\rho}(t)}{\sqrt{1-(\tilde{\rho}(t))^2}}dt+M$$ for $$\tilde{\rho}(t)=\frac{1}{\sinh^{n-1}(r)}\int_0^r\phi(s)\sinh^{n-1}(s)ds.$$ 
Observe that condition \ref{propphi1} and $\phi(r) \le (n-1)\coth r$ imply that $\sup |\tilde{\rho}(t)| < 1$. Hence  
the fact that $v$ is well defined follows from the integrability of $\tilde{\rho}$, that  is proved as Remark \ref{rmk-rhoint}. Therefore we have an upper barrier for the Dirichlet problem \eqref{eq-asymDP}.

We conclude the proof by modifying our function $f$ in \eqref{eq-asymDP} to a new function $\hat{f}\in C^2(\Hn\times \R)$ satisfying:
$$\hat{f}(x,t)=\begin{cases} f(x,t) \text{ if }t\leq v(0)\\
                           g(x,t)\text{ if }t\geq v(0)+1\end{cases}$$
where $g$ satisfies conditions \ref{propphi1} and \ref{propphi2}. Hence, from the previous case, there is $\hat{u}$ solution to the \eqref{eq-asymDP} with $\hat{f}.$ Nevertheless, $\hat{u}$ also satisfies \eqref{eq-asymDP} for $f$ since $\hat{u}$ is bounded by $v\circ r$ and therefore by $v(0).$

\end{proof}

\section{Removable asymptotic singularities}

In this section, we show that there is no isolated singularity on the asymptotic boundary for the solution to \eqref{eq-asymDP}. For that, we study a Dirichlet problem similar to the one studied in section \ref{The_asymptotic_Dirichlet_problem_subsection}, in which we relax the boundary condition: 
\begin{equation}%
\begin{cases}
Q(v)  = f(x,v)\text{ in } \mathbb{H}^n \\
v =\varphi\text{ on }\partial_{\infty} \mathbb{H}^n \backslash \{p_1, \dots, p_k\},
\end{cases}
\label{mingrapheqWithSing}%
\end{equation}
where $\varphi\in
C^{0}(\partial_{\infty}\mathbb{H}^n)$ is a given function, $p_i \in \partial_{\infty} \mathbb{H}^n$ and $f \in C^1(\mathbb{H}^n \times \mathbb{R})$ satisfies conditions \ref{propphi1}, \ref{propphi2} and $f_t(x,t)\le 0$  in $\Hn\times\mathbb{R}$. Using the Scherk type solutions and following the same idea as in \cite{BR}, we prove that a solution to this problem can be extended continuously to
the points $p_i$, that is, such a solution satisfies $v =\varphi$ on $\partial_{\infty} \mathbb{H}^n$.

However, since our Scherk type solutions are not isometric, in contrast with the Scherk solutions used in \cite{BR}, we need some auxiliary results to prove that the solutions are bounded.  For that, remind that $\psi=\psi_S$, defined in Proposition \ref{prop-Psi},  satisfies
$$ \psi_S(d,t) = \Psi ( d - d_S(o),t).$$
where 
$$\Psi(z,t) = \left\{ \begin{array}{rr}  \sqrt{ \phi(|z|) h(t)} & {\rm if} \quad t \ge 0 \\[5pt]
                                         \sqrt{ \phi(|z|) h(0)} & {\rm if} \quad t < 0
\end{array} \right. $$
 and $d(x)=d_S(x)$ is the signed distance function to $S$. Observe that for any totally geodesic hypersphere $S_0$ that contains the point $o$, we have $d_{S_0}(o)=0$ and, therefore, $\Psi(d,t)=\psi_{S_0}(d,t)$. Then, from Proposition \ref{prop-Psi}, $\Psi$ satisfies the conditions (i)-(vi). Indeed, we have:

\begin{lem}
Let $o \in \mathbb{H}^n$ as stated in  \ref{propphi1}. Then there exists a nonnegative $C^1$ function $\Psi: \mathbb{R} \times \mathbb{R} \to \mathbb{R}$ that satisfies \ref{psiDefinitiona} 
-\ref{psiDefinitiong} from Proposition \ref{prop-Psi} and such that
for any totally geodesic hypersphere $S$,
\begin{equation} 
 |f(x,t)| \le \Psi(d_S(x)-d_S(o), t) \text{ for }x\in \mathbb{H}^n \text{ and }  t \ge 0.
\label{PsiBoundsfxt}
\end{equation}
\end{lem}
\noindent 

\begin{rem}
 Remind that for any totally geodesic hypersphere $S$ and any $c\in \R$, according to Corollary \ref{MainTheorem}, there exists a solution $w_{S,c}(d)$ of \eqref{superScherkEDO} with $\psi$ replaced by $\psi_S$.
Defining 
$$g_{S,c}=\frac{w_{S,c}'}{\sqrt{1+(w_{S,c}')^2}},$$
we have that $(w_{S,c},g_{S,c})$ satisfies \eqref{systemForWG} with $\psi$ replaced by $\psi_S$. 
Now we present a result of some uniform bound of $w_{S.c}$.
\label{SolutionOfScherkWithPsiS}
\end{rem}

\begin{lem}
There exist $c_0 >0$ and $d_1>0$ with the following property: if $c \ge c_0$, there exists a constant $M >c$ that depends only on $c$, $n$, and $\Psi$, such that, for any totally geodesic hypersphere $S$, 
$$ w_{S,c}(d) \le M \quad {\rm for} \quad d \ge d_1. $$
\label{AuxiliarLemmaForUniformBoundedness2}
\end{lem}

\begin{proof}
Let $S$ be a totally geodesic hypersphere and $c\ge 0$. Remind, from Theorem \ref{g-Gamma0IsSol}, that
$$ \lim_{d \to 0}g_{S,c}(d) = -1.$$
Using this, \eqref{systemForWG} and $w_S(+\infty)=c$, we conclude that
\begin{equation}
w_{S,c}(d) = c - \int_{d}^{+\infty} \frac{g_{S,c}(t)}{\sqrt{1-g_{S,c
}^{2}(t)}}\; dt \label{wExpressionByGFinal}%
\end{equation}
and
\begin{equation}
g_{S,c}(d) = \frac{1}{\cosh^{n-1} (d)} \left( -1 -
\int_{0}^{d} \psi_S(s,w_{S,c}(s)) \cosh^{n-1} (s) ds \right)  .
\label{gExpressionByWFinal}%
\end{equation}
Observe now that, from Corollary \ref{MainTheorem}, $w_{S,c}$ is decreasing and, therefore, $w_{S,c}(d) \ge c$. Hence, using \eqref{gExpressionByWFinal} and that $\psi_{S}(d,t)$ is nonincreasing in the  variable $t$, we get
\begin{equation}
g_{S,c}(d) \ge \frac{1}{\cosh^{n-1} (d)} \left( -1 -
\int_{0}^{d} \psi_S(s,c) \cosh^{n-1} (s) ds \right).
\label{LowerBoundForGSc}
\end{equation}
From \ref{psiDefinitionef}, there exists $c_0 > 0$ such that
$$ \psi_{S}(d,t) = \Psi(d-d_S(o),t) \le \Psi(0,t) \le \frac{1}{2^{n+1}} $$
for any $d \in \mathbb{R}$ and $t \ge c_0$.
This $c_0$ does not depend on $S$, but only on $\phi(0)$ and $h$. Hence, if $c \ge c_0$, inequality \eqref{LowerBoundForGSc} implies that
$$ g_{S,c}(d) \ge \frac{1}{\cosh^{n-1}(d)} \left( -1 -
\int_{0}^{d} \frac{1}{2^{n+1}} \cosh^{n-1}(s) ds \right). $$
Therefore, using that $\cosh^{n-1} (d) > 4$ for $d \ge 4$ and 
\[
\frac{\displaystyle \int_{0}^{d} \cosh ^{n-1}(s) ds}{\cosh ^{n-1}(d)} <
2^{n-1} \text{ for any } d \ge 0,
\]
we conclude that 
$$ g_{S,c}(d) \ge - \frac{1}{2} \text{ for } d \ge 4 \text{ and } c \ge c_0.$$ 
Moreover, $ g_{S,c}(d) < 0$ according to the proof of Corollary \ref{MainTheorem}. Then $|g_{S,c}(d) | \le 1/2$ for $d \ge 4$ and $c \ge c_0$.
From this and \eqref{wExpressionByGFinal},
\begin{equation}
 w_{S,c}(d) \le c + \frac{2}{\sqrt{3}} \int_{d}^{+\infty} |g_{S,c}(t)|\; dt \text{ for }  d \ge 4 \text{ and } c \ge c_0.
\label{UpperBoundForWscForDBiggerThan4}
\end{equation}
Now using again \eqref{gExpressionByWFinal} and that $\psi_S(d,0) \ge \psi_S(d,t)$ for any $d$, we get
$$ |g_{S,c}(d)| < \frac{1}{\cosh^{n-1}( d)} + \rho(d),$$
where $\rho$ is defined in \eqref{eq-defrho} with $d_0=0$ and $\psi$ replaced by $\psi_S$. Then, from \eqref{UpperBoundForWscForDBiggerThan4} and \eqref{boundForIntegralOfRho}, we have
\begin{align*}
w_{S,c}(d) &\le c + \frac{2}{\sqrt{3}} \int_{d}^{+\infty} \left(\frac{1}{\cosh^{n-1}(s)} + \rho(s)\right) \; ds  \\[5pt]
           &\le c + \frac{2}{\sqrt{3}} \int_{0}^{+\infty} \frac{1}{\cosh^{n-1}(s)} \; ds +  \frac{2^n}{\sqrt{3}(n-1)} \int_0^{+\infty} \psi_S(s,0) \; ds, 
\end{align*} 
for $d \ge 4$ and $c \ge c_0$. Since, by the definition of $\Psi$ and \ref{psiDefinitiong},
$$\int_0^{+\infty} \psi_S(s,0) \; ds =\int_{-d_S (o)}^\infty \Psi (s,0) ds \le \int_{-\infty}^{\infty} \Psi (s,0) ds < \infty,$$ 
the proof follows.

\end{proof}

Now we state a comparison principle for some unbounded domains.

\begin{lem}
Let $U$ be a domain in $\mathbb{H}^n$, possibly unbounded. Suppose that $w_1$ and $w_2$ are respectively a sub and a supersolution of $Q(w(x))  =  \psi(d(x),w(x))$ in $U$ such that
$$ \limsup_{x \to p} w_1(x) \le \liminf_{x \to p} w_2(x) \text{ for any }  p \in \partial U \cup \partial_{\infty} U$$
and,  for some $A < \infty$,
 we have $\limsup_{x \to q} w_1(x) \le A$ for any $q \in \partial U \cup \partial_{\infty} U$.
Then, $w_1 \le w_2$ in $U$.
\label{comparison-principle-for-unbounded}
\end{lem}

Now making some adjustments in the proof of Theorem 1.1 of \cite{BR}, we obtain:
\begin{prop}
If $u \in C^2(\mathbb{H}^n) \cap C^0(\overline{\mathbb{H}^n} \backslash \{p_1, \dots, p_k \})$ is a solution of \eqref{mingrapheqWithSing}, then $u$ is bounded.
\label{boundednessOfSolutionWithSing}
\end{prop}
\begin{proof}
For each $p_i$ let $B_i$ be a totally geodesic hyperball such that $p_j\in \partial_\infty B_i$ if and only if $i=j.$ We can suppose $p_i \in \text{int}\; \partial_{\infty} B_i$ and we denote by $H_i$ the hypersphere that bounds $B_i.$

Since $p_j \not\in \partial_{\infty}H_i$ for any $j \in \{1, \dots, k\}$, $u$ is continuous in $H_i \cup \partial_{\infty} H_i$ and, therefore, is bounded on $H_i$.
Then set
$$ c_1 =\max\{c_0, \sup_{\partial_{\infty} \mathbb{H}^n} \varphi, \sup_{H_1} u, \dots, \sup_{H_k} u\}, $$
where $c_0$ is given in Lemma \ref{AuxiliarLemmaForUniformBoundedness2}. Since $c_1 \ge c_0$, we deduce from Lemma \ref{AuxiliarLemmaForUniformBoundedness2} that there exist $M > c_1$ and $d_1>0$ such that 
\begin{equation}
w_{H,c_1}(d) \le M \text{ for } d \ge d_1,
\label{UniformBoundedNessForSolutions}
\end{equation}
for any totally geodesic hypersphere $H$. Let $H^*$ be some totally geodesic hypersphere contained in $B_1$ such that $dist(H^*,\partial B_1) > d_1$. Hence,\\  $dist(H^*,\mathbb{H}^n \backslash B_1 \cup \dots \cup B_k) > d_1$ and, therefore \eqref{UniformBoundedNessForSolutions} implies that 
$$ w_{H^*,c_1}(d_{H^*}(x)) \le M \text{ for } x \in \mathbb{H}^n \backslash B_1 \cup \dots \cup B_k. $$
Using that $u \le c_1$ on the boundary and asymptotic boundary of $\mathbb{H}^n \backslash B_1 \cup \dots \cup B_k$, the comparison principle (Lemma \ref{comparison-principle-for-unbounded}) and $w_{H^*,c_1} \ge c_1$, we conclude that 
\begin{equation}
u \le w_{H^*,c_1} \le M \text{ in } \mathbb{H}^n \backslash B_1 \cup \dots \cup B_k.
\label{uIsSmallerThanC1}
\end{equation}
 
Now we prove that $u \le M$ in $B_i$. For that, take a sequence of totally geodesic hyperspheres $S_m$ that converges to $p_i$. 
Let $Y_m$ be the connected component of $\mathbb{H}^n \backslash S_m$ whose asymptotic boundary does not contain $p_i$. Observe that $\mathbb{H}^n \backslash B_i \subset Y_m$ for $m$ large and $\cup_m Y_m = \mathbb{H}^n$. Consider the problem
$$\begin{cases}
Q(v\circ d_{S_m})= \psi_{S_m}(d_{S_m},v\circ d_{S_m}) \text{ in } Y_m\\[5pt]%
v = c_1 \text{ on }\partial_{\infty} Y_m\\[5pt]%
v = + \infty\text{ on } S_m,
\end{cases}
$$
where $\psi_{S_m}$ is defined in Proposition \ref{prop-Psi} and in the beginning of this section. According to Remark \ref{SolutionOfScherkWithPsiS}, this problem has a solution $w_{S_m,c_1}(d_{S_m}(x))$. Moreover, from \eqref{UniformBoundedNessForSolutions}, we get
\begin{equation}
 w_{S_m,c_1}(d) \le M \quad {\rm for } \quad d \ge d_1.
 \label{UniformBoundedNessForSolutionsWSm2}
\end{equation} 
From Corollary \ref{MainTheorem}, $w_{S_m,c_1}(d) > c_1$ for any $d$. Hence, $w_{S_m,c_1} > c_1 \ge u$ on $H_i = \partial B_i$, $w_{S_m,c_1} \ge u$ on $\partial_{\infty} ( Y_m \cap B_i)$ and  $w_{S_m,c_1} = +\infty > u$ on $\partial Y_m$. Therefore, Lemma \ref{comparison-principle-for-unbounded} implies that 
$$u(x)\le w_{S_m,c_1}(d_{S_m}(x)) \text{ in }  Y_m \cap B_i \text{ for }m \text{ large.}$$ 
Let $z \in B_i$. Then $z \in Y_m \cap B_i$ for $m$ sufficiently large, since $\cup Y_m = \mathbb{H}^n$. 
Moreover, using that $S_m$ converges to $p_i$, we have $d_{S_m}(z) > d_1$ for $m$ large. Therefore, from \eqref{UniformBoundedNessForSolutionsWSm2}, we conclude that
$$ w_{S_m,c_1}(d_{S_m}(z)) \le M \text{ for } m \text{ large.}$$
Hence
$u \le M$ in $B_i$. From this and \eqref{uIsSmallerThanC1}, it follows that $u \le M$ in $\mathbb{H}^n$. Analogously, one can show that $u$ is bounded from below. This establishes the proof. 
\end{proof}

The proof of the next result follows the same idea as in Theorem 1.1 of \cite{BR}.

\begin{proof}[Proof of Theorem \ref{ContinuityOfSolutionWithSing}]
For $p \in \{p_1, \dots, p_k\}$, we have to prove that 
$$\lim_{x\to p} u(x) =\varphi(p).$$
Observe that $v:=u-\varphi(p)$ is a solution of $Q(v)=\tilde{f}(x,v)$ in $\mathbb{H}^n$ and $v=\varphi-\varphi(p)$ on $\mathbb{H}^n \backslash \{p_1, \dots, p_k \}$, where $\tilde{f}(x,v):=f(x,v+\varphi(p))$ satisfies \ref{propphi1} and \ref{propphi2}. Then, we can suppose w.l.g. that $\varphi(p)=0$.
That is, we need to show that 
$$\lim_{x\to p} u(x) = 0.$$
Let
$$ K = \limsup_{x \to p} u(x) .$$
According to Proposition \ref{boundednessOfSolutionWithSing}, $u$ is bounded from above by some $M,$ so $K\leq M.$ We show now that, for any $\delta > 0$, we have $K \le \delta$. Suppose by contradiction that $K > \delta$, for some $\delta >0$.
Now consider a decreasing sequence $(V_j)$ of totally geodesic hyperballs ``concentric" at $p$
(that is, $p$ is one of the ending point of a geodesic that cross each $\partial V_j$ orthogonally), 
such that $$\bigcap_j \overline{V}_j = \{p\},\quad  \sup_{V_j} u < K + 1/j \quad \text{ and } \quad \sup_{\partial_\infty V_j}\varphi \le \frac{\delta}{2}. $$
For each $j$, let $\tilde{V}_j \subset V_j$ be a totally geodesic hyperball  concentric with $V_j$ at $p$ such that
$$ dist( \partial \tilde{V}_j, \partial V_j ) \ge j \text{ and } \sup_{x \in \tilde{V}_j} u(x) > K -1/j.$$
Then there exists a sequence $(x_j)$  that satisfies $x_j \in \tilde{V}_j$ and
$$ K - 1/j < u(x_j) < K + 1/j .$$
Denote $A = V_1$ and let $T_j:\mathbb{H}^n \to \mathbb{H}^n$ be an isometry  that preserves $p$,  $T_j(\tilde{V}_j) \supset A$ and $y_j:=T_j(x_j) \in \partial A.$  Since $T_j(V_j)$ and $T_j(\tilde{V}_j)$ are totally geodesic hyperballs and $T_j(V_j) \supsetneq T_j(\tilde{V}_j) \supset A$, we have that $\partial_{\infty}A \subset {\rm int} \; \partial_{\infty} T_j(V_j)$ for any $j$.
Observe that $$u_j= u \circ T^{-1}_j$$
 satisfies
\begin{equation}
\sup_{T_j(V_j)}u_j < K + 1/j \text{ and }  u_j(y_j) > K - 1/j,
\label{limiteSuperiorWm}
\end{equation}
and is a solution of 
$$Q(v(y)) = f(T^{-1}_j(y),v(y)),$$
since $Q$ is invariant under isometries. We have also that 
$T_j(V_j)$ is a totally geodesic hyperball and
$u_j \le \delta/2$ on $ \partial_{\infty} (T_j(V_j)) \backslash \{p\}$
since $u=\varphi \le \delta/2$ on $\partial_{\infty} V_j \backslash \{p\}$.
Moreover, using that $A \subset T_j(V_j)$ and $p \not\in \partial_{\infty}( \mathbb{H}^n \backslash A )$, we have that $\partial_{\infty}A \cap \partial_{\infty}( \mathbb{H}^n \backslash A )\subset \partial_{\infty}T_j(V_j)\backslash \{p\}$. Therefore, $u_j \le \delta/2$ on $\partial_{\infty}A \cap \partial_{\infty}( \mathbb{H}^n \backslash A)$. 

For $q \in \partial_{\infty}A \cap \partial_{\infty}( \mathbb{H}^n \backslash A)$, let $B_q$ be a totally geodesic hyperball  disjoint with $V_2$ such that $q \in {\rm int} \partial_{\infty}B_q$ and $B_q \subset T_j(V_j)$ for any $j$. (This is possible since $(V_j)$ is a decreasing sequence, $\partial T_j(V_j)$ is a totally geodesic hypersphere, $dist(\partial T_j(V_j),A) \ge j$ as in \cite{BR} and, then some neighborhood of $\partial_{\infty}A \subset {\rm int} \, \partial_{\infty} T_j(V_j)$ for any $j$). As in Theorem \ref{thm-Scherk},
consider the supersolutions $w_q$ of 
$$\begin{cases}
  Q(v)(y) = f(T^{-1}_j(y) , v(y))\text{ in } B_q \\
v  = +\infty \text{ on } \partial B_q \\
v = \delta/2 \text{ on } {\rm int} \; \partial_{\infty} B_q.\end{cases} $$
Such a problem is solvable according to Theorem \ref{s}, since $f(T^{-1}_j(y) , t)$ satisfies \ref{propphi1} and \ref{propphi2}. 
Since $u_j \le w_q=\delta/2$ on ${\rm int} \; \partial_{\infty} B_q$, Lemma \ref{comparison-principle-for-unbounded} implies that $u_j \le w_q$ in $B_q$. Let $\tilde{B}_q \subset B_q$ be the hyperball with boundary equidistant to $\partial B_q$, for which $w_q < \delta$ in $\tilde{B}_q$. Hence $u_j < \delta$ in $\tilde{B}_q$ and, therefore, 
\begin{equation}
u_j < \delta \text{ in } \tilde{B}
\label{u-jIsSmallerThanDelta}
\end{equation} for any $j$, where
$$ \tilde{B} = \bigcup_{q \in \partial_{\infty}A \cap \partial_{\infty}( \mathbb{H}^n \backslash A)} \tilde{B}_q .$$
Observe that $\tilde{B}$ is a neighborhood of $\partial_{\infty}A \cap \partial_{\infty}( \mathbb{H}^n \backslash A)$ and $\partial A \backslash \tilde{B}$ is compact.

\

Now we prove that there exists some $w$ defined in $\mathbb{H}^n$ that is the limit of some subsequence $(u_j)$ and is a solution of the minimal surface equation. This function is also not constant and satisfies $\max w = K$ contradicting the maximum principle.

First, remind that we have already noted that $T_j(V_j) \supset A$ and $$dist(\partial T_j(V_j),A) \ge j,$$ which implies that ``$T_j(V_j) \to \mathbb{H}^n$".
This means that any compact set $F \subset \mathbb{H}^n$ is contained in $T_j(V_j)$ for large $j$. Since $|u|$ is bounded by $M$, we have $\sup_F |u_j| \le M$, for large $j$.
In fact, this estimative holds in any neighborhood of $F$.
Hence, from Proposition \ref{lem-gradint},
\begin{equation}
\sup_F |\nabla u_j| \le L,
\label{uniform-boundedness-of-u_j}
\end{equation}
where $L$ is some positive constant that does not depend on $j$.
From Arzel\`a-Ascoli, there is some subsequence of $(u_j)$ that converges uniformly in $F$.  
Taking an increasing sequence of compacts sets $F_m$ such that $\bigcup_m F_m =\mathbb{H}^n$ and applying a diagonal argument, we conclude that there exists some subsequence, that we rename by $u_j$, such that $u_j$ converges uniformly in any compact subset of $\mathbb{H}^n$. Let
$$w(x) = \lim_{j\to \infty} u_j(x).$$

From \eqref{uniform-boundedness-of-u_j}, for any bounded open set $U$, we have also that $u_j$ is uniformly bounded in $C^{1}(\overline{U})$. From the linear eliptic PDE theory, $(u_j)$ admits a converging subsequence in $C^{2,\alpha}$ for some $\alpha\in (0,1).$ Let $w$ denote the limit of this subsequence. Again, using a diagonal argument, we have that some subsequence converges to $w$ in $C^{2,\alpha}(\mathbb{H}^n)$. We can denote this subsequence by $u_j$.

Since ``$T^{-1}_j(F) \to \partial_{\infty}\mathbb{H}^n$" as $j \to \infty$ for any compact $F$, condition \ref{propphi1} implies that $f(T^{-1}_j(y), u_j(y)) \to 0$ uniformly in $F$ as $j \to \infty$. Hence, using \eqref{uniform-boundedness-of-u_j} and that $u_j$ converges 
 to $w$ in $C^2,$ 
 we have that $Q(w)=0$. From the classical theory, the graph of $w$ is a minimal surface. Moreover, from $T_j(V_j) \supset F$ for large $j$ and \eqref{limiteSuperiorWm}, it follows that 
$$\sup_{\mathbb{H}^n} w \le K.$$
Now, remind that $y_j=T_j(x_j) \in \partial A$. From \eqref{limiteSuperiorWm} and \eqref{u-jIsSmallerThanDelta}, we conclude that $y_j \in \partial A \backslash \tilde{B}$ if $1/j < K - \delta$. Since $\partial A \backslash \tilde{B}$ is compact, upon passing to a subsequence, $y_j$ converges to some $y \in \partial A \backslash \tilde{B}$. Hence, from \eqref{limiteSuperiorWm} and the fact that $u_j$ converges to $w$ uniformly in compact sets, we have that $w(y) = K$.
Then $y$ is a maximum point of $w$ and, therefore, by the maximum principle, $w$ is constant. However this contradicts that $w(y)=K$ and $w < \delta < K$ in $\tilde{B}$.  From this, we conclude that $K \le \delta$ for any $\delta >0$ and, therefore, $K=0$.

By a similar argument, we can prove that $\liminf_{x \to p} u(x) =0$, proving that $\lim_{x \to p}w(x)=0=\varphi(p)$.
\end{proof}

\bibliography{biblioBCKRT}
\end{document}